\let\ORIlabel\label
\let\ORIrefstepcounter\refstepcounter
\AddToHook{package/hyperref/before}{
   \let\label\ORIlabel 
   \let\refstepcounter\ORIrefstepcounter}
\documentclass[onefignum,onetabnum]{siamart190516}
\usepackage[T1]{fontenc}
\def\CT{{\mathcal T}}

\newcommand{\C}{\mathbb{C}}
\newcommand{\K}{\mathbb{K}}
\newcommand{\ten}[1]{\mathbf{#1}}
\usepackage{mathtools} 
\usepackage{booktabs}

\usepackage{amsmath,amssymb,amsfonts,latexsym,stmaryrd}
\usepackage{graphicx,float} 
\usepackage{mathrsfs} 
\usepackage{subcaption} 

\usepackage{color}
\usepackage{setspace} 
\usepackage{lipsum}
\usepackage{graphicx,float}
\usepackage{color}
\usepackage{epstopdf}
\usepackage{multirow}
\usepackage{url}
\usepackage{diagbox}
\def\O{\Omega}

\renewcommand\sp{\mathop{\mathrm{Sp}}\nolimits}

\usepackage{booktabs}
\usepackage{float}

\def\hdel{\widehat{\delta}}

\newcommand\cT{\mathcal{T}}
\def\CT{{\mathcal T}}

\newcommand\R{\mathbb{R}}

\renewcommand\O{\Omega}

\newcommand{\vertiii}[1]{{\left\vert\kern-0.25ex\left\vert\kern-0.25ex\left\vert #1 
		\right\vert\kern-0.25ex\right\vert\kern-0.25ex\right\vert}}

\renewcommand\sp{\mathop{\mathrm{sp}}\nolimits}

\newsiamremark{remark}{Remark}
\newsiamremark{hypothesis}{Hypothesis}
\crefname{hypothesis}{Hypothesis}{Hypotheses}
\newsiamthm{problem}{Problem}
\newsiamthm{claim}{Claim}

\headers{Neutron transport equations}{Nicol\'as A. Barnafi,  Felipe Lepe and Francisca A. Mu\~{n}oz}

\title{Finite element analysis  of an  eigenvalue problem arising from neutron transport \thanks{
\funding{NAB was supported by Centro de Modelamiento Matematico (CMM), Proyecto Basal FB210005, and by ANID Postdoctoral Proyecto 3230326. FL was supported by Universidad del B\'io-B\'io through  Proyecto Regular RE2514703.
}}}

\author{Nicol\'as Barnafi\thanks{
Instituto de Ingenier\'ia Matem\'atica y Computacional \& Facultad de Ciencias Biol\'ogicas, Pontificia Universidad
Cat\'olica de Chile, Av Vicu\~{n}a Mackenna 4860, Santiago, Chile, and Center for Mathematical Modeling, Santiago, Chile, \email{nicolas.barnafi@uc.cl}.}
\and Felipe Lepe\thanks{GIMNAP-Departamento de Matem\'atica, Universidad del B\'io - B\'io, Casilla 5-C, Concepci\'on,
Chile, \email{flepe@ubiobio.cl}.}
\and Francisca Mu\~{n}oz\thanks{Instituto de Ingenier\'ia Matem\'atica y Computacional, Pontificia Universidad
Cat\'olica de Chile, Av Vicu\~{n}a Mackenna 4860, Santiago, Chile, and School of Industrial and Information Engineering, Politecnico di Milano, Piazza Leonardo da Vinci, 32 - 20133 Milano,  \email{fmur@uc.cl}.}}

\usepackage{amsopn}

\begin{document}

\maketitle

\begin{abstract}
In two and three dimensions, we analyze a finite element method to approximate the solutions of an eigenvalue problem arising from neutron transport. We derive the eigenvalue problem of interest, which results to be non-symmetric. Under a standard finite element approximation based on piecewise polynomials of degree $k\geq 1$, and under the framework of the compact operators theory, we prove convergence and error estimates of the proposed method. We report a series of numerical tests in order confirm the theoretical results.
\end{abstract}

\begin{keywords}
	Non-symmetric eigenvalue problem,  convergence, a priori error estimates
\end{keywords}

\section{Introduction}\label{sec:intro}
Neutron transport models describe the behavior of neutrons within a given domain and their interactions with the medium, involving parameters such as direction, energy, position and time.  These models are formulated from a balance of neutron density in time, which results in a formulation in space, time, and energy. To reduce the complexity given by the energy domain, the resulting system is commonly quantized into energy groups, with the most common model being the one with two groups (fast and thermal). We use that model in this work. We provide a systematic and detailed derivation of the multigroup equation, which is not available in the literature to the best of the authors knowledge. These models have applications in the functioning of pressurized water reactors, boiling water reactors, natural uranium gas graphite reactors, and many more. We refer to the reader to  \cite{reactorbook} for details on these subjects. Still, their mathematical and approximation properties remains understudied.

The time dependent formulation of the Neutron Transport Equation involves several variables associated with neutron production and loss rates mechanisms. Although the model is inherently complex, under nuclear reactors conditions it can be simplified by restricting the analysis to a bounded spatial domain with specific conditions. By studying the transient problem, we see that the neutron flux evolves in time until it reached an equilibrium state. At this point the ratio between neutron production and neutron losses stabilizes, even if it was not initially balanced. At equilibrium, this equation reduces to a stationary balance between neutron gains and losses. To achieve this equilibrium the fission source term must be scaled by a multiplication factor. This leads naturally to an eigenvalue problem that represents the system at its stationary balance, where the multiplication factor will be the eigenvalue of this problem. This eigenvalue problem results to be naturally non-selfadjoint, and its analysis must be performed according to this feature when numerical methods such as finite elements are considered.
Models related to neutron transport involve several variables in their governing PDEs. Under certain assumptions on the solutions, these models can be reduced to an eigenvalue problem. This eigenvalue problem is inherently non-selfadjoint, and its analysis must therefore account for this property, particularly when numerical methods such as finite elements are employed.

The analysis of non-selfadjoint eigenvalue problems is a well established topic in the literature, and we refer to \cite{MR366029,MR447842} as classic references related on this topic. These articles show how non-symmetric eigenvalues can be analyzed, and have led  to several  applications such as \cite{MR4368386, MR4930005,MR3259022,MR2845628,MR4728079,MR4846353}, among others. Here, the lack of symmetry requires handling the eigenvalue in a proper manner in order to analyze the numerical methods and the corresponding convergence and error estimates.

The eigenvalue problem arising from neutron transport is non-selfadjoint and hence in this work we perform its analysis within that framework. The main ingredient that renders the model suitable for it is the compactness of both the solution operator and its adjoint, which is shown using well known regularity results of the Laplace problem. This leads to a straightforward numerical analysis based on the convergence in norm. The analysis establishes that the model can be accurately approximated with piecewise polynomials of degree $k\geq1$, and indeed this suffices to avoid generating spurious eigenvalues. Despite the fact that the engineering literature is not scarce on simulations of neutron transport under different realistic physical configurations, to the best of our knowledge the spectrum of the operator associated to the neutron transport eigenvalue problem has not been analyzed under a formal mathematical basis, and also, there are not computational results that can indicate accurately the spectrum of such an operator, namely the physical eigenvalues. Hence, this work is a novel effort to advance the mathematical structure of this problem, with the aim to continue the research on this topic with other numerical methods.

The organization of the manuscript is the following: In Section \ref{sec:model_problem} we derive the eigenvalue problem. Here, the derivation is performed under the physical problem related to the transport equation of neutrons on a reactor. One of the main ingredients for the analysis is the \emph{source problem}, which we define and analyze in Section~\ref{sec:source_problem}.  In Section~\ref{sec:eigen_problem}, we perform the analysis of the eigenvalue problem. More precisely, we present the solution operator, the adjoint source problem and show the required compactness results. Section \ref{sec:fem} contains the discrete framework in which our analysis is supported. We present the discrete eigenvalue problem and the discrete solution operators. Also, the convergence between the discrete and continuous operators is proved and and the corresponding implications are derived, such as spectral convergence, the spurious free result, and error estimates. Finally in Section \ref{sec:numerics} we report several numerical tests in two and three dimensions in order to confirm our theoretical results.

\subsection*{Notation}
We consider $L^2(\O, \K)$ the space of square-integrable functions defined on a Lipschitz bounded domain $\O$ with values on a field $\K$ ($\R$ or $\C$). When $\K$ is omitted, it implies $\R$. This space is endowed with the standard inner product
    $$
        (f, g)_0 \coloneqq (f,g)_{0,\O} = \int_\O f \overline g\,dx,
    $$ 
where $\overline{(\cdot)}$ denotes complex conjugation. We also denote by $H^1(\O)$ and $H_0^1(\O,\mathbb{C})$ the standard $H^1$ Sobolev space of real and complex functions respectively, endowed with the usual inner product given by
    $$
        (v, w)_1 \coloneqq (v,w)_{1,\O}=\int_{\O}v\overline{w}+\int_{\O}\nabla v\cdot\nabla\overline{w}\quad\forall v,w\in H_0^1(\O,\mathbb{C}).
    $$
 The norm induced by the inner products defined above will be denoted by $\|\cdot\|_{0,\O}$ and $\|\cdot\|_{1,\O}$ respectively.

\section{The model problem}
\label{sec:model_problem}

Neutron Transport modeling consist in determining the number of neutrons within a given bounded domain $\omega\subset \R^d, d\in\{2,3\}$. We introduce the neutron density $n(\vec{r},{\epsilon}, \vec{\vartheta},t)$, at position $r\in \O$, energy $\epsilon \in [E_0, E_2]$, direction $\vec \vartheta \in [0,2\pi]$  and time $t\in (0,T] $, which represents the number of neutrons per unit volume. We also introduce the neutron flux $\phi(\vec{r},{\epsilon}, \vec{\vartheta},t)$, which represents the number of neutrons passing through a unit area per unit time. For simplicity, we will neglect the direction dependence as it is of minor importance \cite{booknuc3}, as we are only interested in the amount of neutrons and not their direction. We thus redefine the neutron density and scalar flux as 
    \begin{equation*}
        n(\vec{r},\epsilon,t)=\int_{4\pi} n(\vec{r},\epsilon,\vec{d},t)\,d\vec{\vartheta}\quad
\text{and}\quad         \phi(\vec{r},\epsilon,t)=\int_{4\pi} \phi(\vec{r},\epsilon,\vec{d},t)\,d\vec{\vartheta},
    \end{equation*}
respectively. Note that integrals on "$4\pi$" are common notation in this field and they refer to integration on the directions domain. To model the neutron transport equation, we express the time variation of the neutron density, for the neutron population with a fixed energy $\epsilon$, as the difference between the neutron generation rate and neutron loss rate in a certain fixed volume \cite{booknuc3} 
\begin{equation*}
    \frac{\partial n({\epsilon},t)}{\partial t} \coloneqq \dot{N}_\text{Generation}(\epsilon,t)-\dot{N}_\text{Loss}(\epsilon,t).
\end{equation*}
Here $n$ does not depend on $\vec{r}$, since we are not interested in how many neutrons there are at each position, but only in the total number of neutron within the domain. Explicit expressions for each contribution are provided in the following sections.

\subsection{Generation rate}

We decompose the generation rate into fission and in-scattering as
\begin{equation*}
    \dot{N}_{\text{Generation}}(\epsilon,t) \coloneqq  \dot{N}_{\text{Fission}}(\epsilon,t)+ \dot{N}_{\text{In-Scattering}}(\epsilon,t).
\end{equation*}

\paragraph{Fission}
Consider the functions $\nu(\epsilon)$ which denotes the average number of neutrons produced per fission event at a certain energy, $\chi(\epsilon)$ as the birth neutron spectrum, which corresponds to the probability distribution of the energy at which neutrons are born, and   $\Sigma_f(\epsilon)$ which represents the nuclear fission cross section and gives the probability that a neutron at a certain energy undergoes a fission event. Leveraging these definitions,  $\Sigma_f(\epsilon)\phi(\vec{r},\epsilon, t)$ gives the fission reaction rate. Finally, multiplying the fission reaction rate by $\nu(\epsilon)$ and integrating over the energy and position, we obtain the total neutron generation rate by fission \cite{booknuc3},
\begin{equation*}
\int_{\Omega}\int_{E}\nu(\epsilon')\Sigma_f(\epsilon')\phi(\vec{r},\epsilon',t)\,d\epsilon'\,d\vec{r}.
\end{equation*}
Multiplying this by $\chi(\epsilon)$ gives us the total neutron generation rate by fission for neutrons at a certain energy $\epsilon$
\begin{equation*}
    \dot{N}_{\text{Fission}}(\epsilon,t) \coloneqq\chi(\epsilon)\int_{\Omega}\int_{E}\nu(\epsilon')\Sigma_f(\epsilon')\phi(\vec{r},\epsilon',t)\,d\epsilon'\,d\vec{r}.
\end{equation*}

\paragraph{In-Scattering}
This phenomenon occurs when a neutron with energy $\epsilon'$  and direction $\vec{\vartheta}'$ 
changes its energy to $\epsilon$ 
and its direction to $\vec{\vartheta}$ 
as a result of a collision. To model this process, we introduce a scattering cross section $\Sigma_s(\epsilon',\vec{\vartheta}'\to\epsilon,\vec{\vartheta})$ that represents the probability that a neutron will undergo such a state change. As before, we consider only the average among all directions, thus we set
\begin{equation*}
\Sigma_s(\epsilon'\to\epsilon) \coloneqq \int_{4\pi}\int_{4\pi}\Sigma_s(\epsilon',\vec{\vartheta}'\to\epsilon,\vec{\vartheta})\,d\vec{\vartheta}\,d\vec{\vartheta'}.
\end{equation*}
By multiplying the scattering cross section $\Sigma_s(\epsilon'\to\epsilon) $ by the neutron flux $\phi(\vec{r},\epsilon')$, we obtain the in-scattering generation rate for neutrons with initial energy $\epsilon'$ that scatter into energy $\epsilon$. Integrating this term over energy $\epsilon'$ and position $\vec{r}$ in the domain yields the total number of neutrons with energy $\epsilon$ generated by scattering \cite{booknuc},

\begin{equation*}
 \dot{N}_{\text{In-Scattering}}(\epsilon,t) \coloneqq\int_{\Omega}\int_E\Sigma_s(\epsilon'\to\epsilon) \phi(\vec{r},\epsilon',t)\,d\epsilon'\,d\vec{r}. 
\end{equation*}

\subsection{Loss rate} 
The neutron loss rate is given by
\begin{equation*}
     \dot{N}_{\text{Loss}}(\epsilon,t) \coloneqq \dot{N}_{\text{Leakage}}(\epsilon,t)+ \dot{N}_{\text{Absorption}}(\epsilon,t) +\dot{N}_{\text{Out-Scattering}}(\epsilon,t).
\end{equation*}

\paragraph{Leakage}
This occurs when the neutron flux exits the domain. We introduce the neutron current $\boldsymbol{J}(\vec{r},\epsilon,t)$ as the vector quantity of neutrons passing per unit area and unit time for a specific direction $\vec\vartheta$. 

At a specific place in the domain boundary, energy and time, the leakage can be expressed as $\boldsymbol{J}(\vec{r},\epsilon,t)\cdot \vec{n}$, where $\vec{n}$ is the outward unit normal vector. Integrating over the  boundary of the domain and energy gives us the total leakage
\begin{equation*}
     \int_{\partial \Omega} \int_E\boldsymbol{J}(\vec{r},\epsilon,t)\cdot \boldsymbol{n}\,d\epsilon d{S}.
\end{equation*}
Using Fick's Law, as in \cite{booknuc}, we write the following relationship between the neutron current density vector and the neutron scalar flux,
\begin{equation*}
    \boldsymbol{J}(\vec{r},\epsilon,t) \coloneqq -D(\epsilon)\nabla \phi(\vec{r},\epsilon,t).
\end{equation*}
Using divergence theorem it is possible to write the following equivalence
\begin{equation*}
    \int_{\partial \Omega}\boldsymbol{J}(\vec{r},\epsilon,t)\cdot \boldsymbol{n}\,d{S} = - \int_{\Omega}\nabla_r\cdot\boldsymbol{J}(\vec{r},\epsilon,t)\,d\Omega= \int_{ \Omega} \nabla_r \cdot D(\epsilon)\nabla_r\phi(\vec{r},\epsilon,t)\,d\vec{r},
\end{equation*}\\
so that the total leakage for neutrons with energy $\epsilon$ is,
\begin{equation*}
     \dot{N}_{\text{Leak}}(\epsilon,t) \coloneqq \int_{ \Omega} \nabla \cdot D(\epsilon)\nabla\phi(\vec{r},\epsilon,t)\,d\epsilon d\vec{r}.
\end{equation*}

\paragraph{Absorption}
To describe the absorption rate we use the absorption cross section $\Sigma_a(\epsilon)$ that represents the probability of a neutron with energy $\epsilon$ being absorbed by a nucleus. Multiplying the neutron flux by the absorption cross section and integrating over the domain yields the total absorption rate for neutron with energy $\epsilon$:
\begin{equation*}
\dot{N}_{\text{Absorption}}(\epsilon,t) \coloneqq \int_{\Omega}\int_{E}\Sigma_a(\epsilon)\phi(\vec{r},\epsilon,t)\,d\vec{r}.
\end{equation*}
\paragraph{Out-scattering}
Similar to the In-Scattering phenomenon, in the Out-Scattering we quantify the amount of neutrons that started with energy $\epsilon$ but scattered into any other energy level $\epsilon'$. To obtain the out-scattering loss rate, we multiply the scattering cross section $\Sigma_s(\epsilon'\to\epsilon)$ by the flux $\phi(\vec{r},\epsilon')$ and integrate over the energy $\epsilon'$ and position $\vec{r}$. This gives us the total number of neutrons that started with energy $\epsilon$ but scattered to another energy level as
\begin{equation*}
 \dot{N}_{\text{Out-Scatteringg}} (\epsilon,t)\coloneqq\int_{\Omega}\int_{E}\Sigma_s(\epsilon\to\epsilon') \phi(\vec{r},\epsilon, t)\,d\epsilon'\,d\vec{r}.
\end{equation*}

\subsection{Multi-group equation at steady state}
Adding the previous generation and loss rates, we obtain the following balance equation for neutron density, at a fixed energy $\epsilon$:
\begin{align*}
    \frac{\partial n(\epsilon,t)}{\partial t} &= \chi(\epsilon)\int_{\Omega}\int_{E}\nu(\epsilon')\Sigma_f(\epsilon')\phi(\vec{r},\epsilon',t)\,d\epsilon'\,d\vec{r} \\ &\quad +\int_{\Omega}\int_{E}\Sigma_s(\epsilon'\to\epsilon)\phi(\vec{r},\epsilon',t)\,d\epsilon'\,d\vec{r}  \\&\quad -   \int_{ \Omega} \nabla \cdot D(\epsilon)\nabla\phi(\vec{r},\epsilon,t)\,d\vec{r}
    \\&\quad - \int_{\Omega}\Sigma_a(\epsilon) \phi(\vec{r},\epsilon,t)\,d\epsilon\,d\vec{r}\\&\quad -\int_{\Omega}\int_{E}\Sigma_s(\epsilon\to\epsilon')\phi(\vec{r},\epsilon',t)\,d\epsilon'\,d\vec{r}.
\end{align*}

To formulate the steady state equation, we simply eliminate the time dependent term. In addition, we must introduce the multiplication factor $k$, which divides the fission side of the equation to ensure a critical state of the system by adjusting the fission term, so that there is a balance between the neutron generation rate and neutron loss rate. This results in the following:
\begin{multline}
\label{mult-factor}
    \dot{N}_{\text{Absorption}}(\epsilon,t) +  \dot{N}_{\text{Leakage}}(\epsilon,t) +\dot{N}_{\text{Out-Scattering}}(\epsilon,t)\\
    =  \dot{N}_{\text{In-Scattering}}(\epsilon,t) + \frac{1}{k} \dot{N}_{\text{Fission}}(\epsilon,t).
\end{multline}
If $k=1$ the system is critical and the amount of neutrons in the system is constant.
A value of $k<1$ means that the system is subcritical and the neutron chain reaction is decreasing. A value $k>1$ means that the system is supercritical and the neutron chain reaction is increasing. Finally, a localization argument yields the local form of \eqref{mult-factor} as follows:
\begin{equation}\label{punctual-eq}
\begin{aligned}
&\Sigma_a(\epsilon) \phi(\vec{r},\epsilon)+ \nabla \cdot D(\epsilon)\nabla\phi(\vec{r},\epsilon)+\int_{E}\Sigma_s(\epsilon\to\epsilon')\phi(\vec{r},\epsilon)\,d\epsilon'=\\ &\int_{E}\Sigma_s(\epsilon'\to\epsilon)\phi(\vec{r},\epsilon')\,d\epsilon'+ \chi(\epsilon)\frac{1}{k}\int_{E}\nu(\epsilon')\Sigma_f(\epsilon') \phi(\vec{r},\epsilon')\,d\epsilon'.
\end{aligned}
\end{equation}
This equations represents the neutron flux balance for a given fixed energy $\epsilon\in [E_0,E_2]$. 

Since the equation for a continuous energy domain is difficult to use in practice, we will work with two energy groups, $[E_0,E_1]$ and $(E_1, E_2]$. Naturally, this work can be easily extended to several energy levels.
We thus define $\phi_1$ and $\phi_2$ as follows:
\begin{equation*}
\phi_1(\vec{r}) \coloneqq \displaystyle\int_{E_{1}}^{E_2} \phi(\vec{r}, \epsilon) \, d\epsilon\quad ,\quad \phi_2(\vec{r}) \coloneqq \int_{E_{0}}^{E_1} \phi(\vec{r}, \epsilon) \, d\epsilon.
\end{equation*}
Here, $\phi_1$ corresponds to the neutron group with higher energy, referred to as the \emph{Fast group} and $\phi_2$ corresponds to the group with lower energy referred to as the \emph{Thermal group}. 

We will redefine the cross sections for each group as a weighted average, 

\begin{subequations}\label{constant-cross}
\begin{footnotesize}
\begin{align}
    \Sigma_{a1} &\coloneqq \frac{\displaystyle\int_{E_1}^{E_2}\Sigma_a(\epsilon) \phi(\vec{r},\epsilon)\,d\epsilon}{\displaystyle\int_{E_1}^{E_2}\phi(\vec{r},\epsilon)\,d\epsilon}, &
    \Sigma_{a2} &\coloneqq \frac{\displaystyle\int_{E_0}^{E_1}\Sigma_a(\epsilon) \phi(\vec{r},\epsilon)\,d\epsilon}{\displaystyle\int_{E_0}^{E_1}\phi(\vec{r},\epsilon)\,d\epsilon}, \\
    D_1 &\coloneqq \frac{\displaystyle\int_{E_1}^{E_2}D(\epsilon) \nabla \phi(\vec{r},\epsilon)\,d\epsilon}{\displaystyle\int_{E_1}^{E_2}\nabla\phi(\vec{r},\epsilon)\,d\epsilon}, &
    D_2 &\coloneqq \frac{\displaystyle\int_{E_0}^{E_1}D(\epsilon)\nabla\phi(\vec{r},\epsilon)\,d\epsilon}{\displaystyle\int_{E_0}^{E_1}\nabla\phi(\vec{r},\epsilon)\,d\epsilon}, \\
    \nu_1\Sigma_{f1} &\coloneqq \frac{\displaystyle\int_{E_1}^{E_2}\nu(\epsilon)\Sigma_f(\epsilon) \phi(\vec{r},\epsilon)\,d\epsilon}{\displaystyle\int_{E_1}^{E_2}\phi(\vec{r},\epsilon)\,d\epsilon}, &
    \nu_2\Sigma_{f2} &\coloneqq \frac{\displaystyle\int_{E_0}^{E_1}\nu(\epsilon)\Sigma_f(\epsilon) \phi(\vec{r},\epsilon)\,d\epsilon}{\displaystyle\int_{E_0}^{E_1}\phi(\vec{r},\epsilon)\,d\epsilon}, \\
    \Sigma_{1\to2} &\coloneqq \frac{\displaystyle\int_{E_1}^{E_2}\int_{E_0}^{E_1}\Sigma_s(\epsilon\to\epsilon') \phi(\vec{r},\epsilon)\,d\epsilon'\,d\epsilon}{\displaystyle\int_{E_1}^{E_2}\phi(\vec{r},\epsilon)\,d\epsilon}, &
    \Sigma_{2\to1} &\coloneqq \frac{\displaystyle\int_{E_0}^{E_1}\int_{E_1}^{E_2}\Sigma_s(\epsilon'\to \epsilon) \phi(\vec{r},\epsilon')\,d\epsilon\,d\epsilon'}{\displaystyle\int_{E_0}^{E_1}\phi(\vec{r},\epsilon')\,d\epsilon'}, \\
    \Sigma_{1\to1} &\coloneqq \frac{\displaystyle\int_{E_1}^{E_2}\int_{E_1}^{E_2}\Sigma_s(\epsilon\to\epsilon') \phi(\vec{r},\epsilon)\,d\epsilon'\,d\epsilon}{\displaystyle\int_{E_1}^{E_2}\phi(\vec{r},\epsilon)\,d\epsilon}, &
    \Sigma_{2\to2} &\coloneqq \frac{\displaystyle\int_{E_0}^{E_1}\int_{E_0}^{E_1}\Sigma_s(\epsilon'\to \epsilon) \phi(\vec{r},\epsilon')\,d\epsilon\,d\epsilon'}{\displaystyle\int_{E_0}^{E_1}\phi(\vec{r},\epsilon')\,d\epsilon'}.
\end{align}
\end{footnotesize}
\end{subequations}
In the multigroup model, we assume that all neutrons born from fission have enegry above $E_1$. Consequently, the probability that a neutron that is born by fission belongs to the fast group is one, and that it belongs to the thermal group is zero. We write this assumption as follows:
\begin{equation*}
\int_{E_0}^{E_1}\chi(\epsilon)\,d\epsilon= 0\quad \text{and}\quad \int_{E_1}^{E_2}\chi(\epsilon)\,d\epsilon=1.
\end{equation*}
\subsubsection{Fast Group}
To obtain the equation corresponding to the Fast Group, we integrate \eqref{punctual-eq} over the Fast Group energy domain $[E_1, E_2]$,
\begin{equation*}
\begin{aligned}
&\int_{E_1}^{E_2}\Sigma_a(\epsilon) \phi(\vec{r},\epsilon)\,d\epsilon+ \int_{E_1}^{E_2}\nabla_r \cdot D(\epsilon)\nabla_r\phi(\vec{r},\epsilon)\,d{\epsilon}+\int_{E_1}^{E_2}\int_{E}\Sigma_s(\epsilon\to\epsilon')\phi(\vec{r},\epsilon)\,d\epsilon'd\epsilon=\\ &\int_{E_1}^{E_2}\int_{E}\Sigma_s(\epsilon'\to\epsilon)\phi(\vec{r},\epsilon')\,d\epsilon'+ \int_{E_1}^{E_2}\chi(\epsilon)\,d\epsilon\,\frac{1}{k}\int_{E}\nu(\epsilon')\Sigma_f(\epsilon') \phi(\vec{r},\epsilon')\,d\epsilon'.
\end{aligned}
\end{equation*}We then simplify this equation using the group averaged expressions \eqref{constant-cross} to derive the following equation:
\begin{equation}
\label{fast}
    -\nabla\cdot(D_1\nabla \phi_1) + (\Sigma_{a1}+\Sigma_{1\to2})\phi_1 = \frac{1}{k}(\nu_1\Sigma_{f1}\phi_1+\nu_2\Sigma_{f2}\phi_2).
\end{equation}
We report the detailed computation of the averaged quantities in Appendix \ref{appendix-a}.
\subsubsection{Thermal Group}
To obtain the equation related  to the Fast Group, we integrate \eqref{punctual-eq} over the Thermal Group energy domain $[E_0, E_1]$,
\begin{equation*}
\begin{aligned}
&\int_{E_0}^{E_1}\Sigma_a(\epsilon) \phi(\vec{r},\epsilon)\,d\epsilon+ \int_{E_0}^{E_1}\nabla \cdot D(\epsilon)\nabla\phi(\vec{r},\epsilon)\,d{\epsilon}+\int_{E_0}^{E_1}\int_{E}\Sigma_s(\epsilon\to\epsilon')\phi(\vec{r},\epsilon)\,d\epsilon'd\epsilon=\\ &\int_{E_0}^{E_1}\int_{E}\Sigma_s(\epsilon'\to\epsilon)\phi(\vec{r},\epsilon')\,d\epsilon'+ \int_{E_0}^{E_1}\chi(\epsilon)\,d\epsilon\,\frac{1}{k}\int_{E}\nu(\epsilon')\Sigma_f(\epsilon') \phi(\vec{r},\epsilon')\,d\epsilon'.
\end{aligned}
\end{equation*}
We then simplify this equation using the group averaged expressions \eqref{constant-cross} to derive the following equation:
\begin{equation}
\label{thermal}
    -\nabla\cdot(D_2\nabla \phi_2) + \Sigma_{a2}\phi_2-\Sigma_{1\to2}\phi_1 = 0.
\end{equation}
We report the detailed computation of the averaged quantities in Appendix \ref{appendix-b}.
Coupling equations \eqref{fast} and \eqref{thermal}, we obtain the following multigroup equation
\cite{booknuc2},
\begin{subequations}\label{eq:neutron-transport-strong}
   \begin{align}
        -\nabla (D_1 \nabla \phi_1) + (\Sigma_{a1} + \Sigma_{1 \rightarrow 2}) \phi_1 &= \frac{1}{k}\left(\nu_1\Sigma_{f1} \phi_1+\nu_2\Sigma_{f2}\phi_2\right)  &\text{in } \Omega, \label{eq:strong1}\\
        -\nabla (D_2 \nabla \phi_2) + \Sigma_{a2}\phi_2  -\Sigma_{1 \rightarrow 2} \phi_1 &= 0  &\text{in } \Omega. \label{eq:strong2}\\
        \frac{\partial \phi_1}{\partial n} = -\alpha\phi_1,\qquad &  \frac{\partial \phi_2}{\partial n} = -\alpha\phi_2 &\text{in }\partial \Omega ,
    \end{align}
\end{subequations}
where $\alpha$ depends on the physical properties of the reactor and neutron medium \cite{articlenuc}.
\begin{remark}
On the forthcoming sections, we will analyze our problem under a suitable functional framework, It is important to remark the analysis of eigenvalue problems need a complete knowledge of the associated  source problems. According to this, we have to clarify the following: for the source problem it is enough to consider real Hilbert spaces, whereas for the eigenvalue problem we need complex spaces. This is due to the non-symmetry of the eigenvalue problem. Hence, and to simplify the presentation of the material, we define the following spaces $$V\coloneqq H^1(\Omega, \C)\times H^1(\Omega, \C),\quad\text{and}\quad Q\coloneqq L^2(\Omega,\C)\times L^2(\Omega, \C).$$

We will use these definitions  throughout all the manuscript, but taking into account that when we solve the source problem,  the spaces must always be defined on real numbers.

\end{remark}

\subsection{Variational Formulation}

By considering test functions $(v_1,v_2)\in V$, we multiply \eqref{eq:strong1} and \eqref{eq:strong2} and integrate by parts to obtain the following problem: Find $\lambda$ in $\C$ and $(0,0)\neq (\phi_1, \phi_2) \in V$ such that
\begin{multline}\label{eq:neutron-transport-weak}
    D_1\int_\O \nabla \phi_1\cdot \nabla v_1 + (\Sigma_{a1} + \Sigma_{1 \rightarrow 2}) \phi_1v_1 \,d\O \\
    + D_2\int_\O \nabla \phi_2\cdot \nabla v_2 + \Sigma_{a2}\phi_2v_2  -\Sigma_{1 \rightarrow 2} \phi_1 v_2\,d\O  \\
    +\int_{\partial \O}  \alpha_1\phi_1v_1 + \alpha_2\phi_2v_2\,dS = \lambda\int_\O \nu_1\Sigma_{f1}\phi_1 + \nu_2\Sigma_{f2}\phi_2\,d\O,
\end{multline}
for all $(v_1, v_2)\in V$. Let us define  the sesquilinear forms $a:V\times V\to \C$ and $b:V\times V\to \C$ given by: 
    \begin{align*}
       a((\phi_1,\phi_2),(v_1,v_2)) &= \int_\O D_1 \nabla \phi_1 \cdot \nabla \overline{v_1} + (\Sigma_{a1} + \Sigma_{1 \rightarrow 2}) \phi_1\overline{v_1} \,d\O \\
       &\quad +  \int_\O D_2 \nabla \phi_2\cdot \nabla \overline{v_2} + \Sigma_{a2}\phi_2\overline{v_2}  -\Sigma_{1 \rightarrow 2} \phi_1 \overline{v_2}\,d\O\\
       &\quad+\int_{\partial \O}  \alpha_1\phi_1\overline{v_1}+\alpha_2\phi_2\overline{v_2}\,dS  \\
       b((\phi_1,\phi_2),(v_1,v_2)) &= \int_\O(\nu_1\Sigma_{f1} \phi_1+\nu_2\Sigma_{f2}\phi_2)\overline{v_2}\,d\O.
    \end{align*}
We now write the abstract variational form of \eqref{eq:neutron-transport-strong} as follows: Find $\lambda$ in $\C$ and $(\phi_1, \phi_2)\in V$ such that
    \begin{equation}\label{eq:neutron-transport-weak-abstract}
        a((\phi_1, \phi_2), (v_1, v_2)) = \lambda b((\phi_1, \phi_2), (v_1, v_2)) \qquad\forall (v_1, v_2) \in V,
    \end{equation}
with $\lambda\coloneqq\dfrac{1}{k}$. We highlight that the smallest $\lambda$ is the most physically relevant eigenvalue, and it is always in $\R$. This fact is commonly acknowledged by the community but does not have a rigorous proof to the best of our knowledge. We establish this result in the following lemma.

\begin{remark}
    We note that this model is commonly framed with values in $\R$. Still, given its non-symmetric nature, both the eigenvalues and eigenvectors can belong also to $\C$. For this reason, we extend this formulation in Section~\ref{sec:source_problem} to $\C$ and continue the analysis within that framework.
\end{remark}

\section{The source problem}\label{sec:source_problem}
The main ingredient of the convergence theory for eigenvalue problems is the compactness of the solution operator. This is defined from the source problem, given by fixing a right hand side in the complex form of \eqref{eq:neutron-transport-weak-abstract}. In this section we provide all the theoretical ingredients for such an analysis.

Consider two functions $(f_1, f_2)\in Q\coloneqq L^2(\Omega)\times L^2(\Omega)$. The source problem is given by: Find $(\widetilde\phi_1, \widetilde\phi_2)$ in $V$ such that
    \begin{equation}\label{eq:source-problem}
        a((\widetilde\phi_1,\widetilde\phi_2),(v_1,v_2)) = b((f_1,f_2), (v_1,v_2)) \qquad\forall (v_1,v_2)\in V, 
    \end{equation}
where $a: V\times V \to\C$ and $b:Q\times V \to\C$ are two sesquilinear forms defined by 
\begin{align*}
    a((\widetilde\phi_1,\widetilde\phi_2),(v_1,v_2)) 
    &= \int_\Omega D_1 \nabla \widetilde\phi_1 \cdot \nabla \overline{v}_1 
    + (\Sigma_{a1} + \Sigma_{1 \rightarrow 2})\, \widetilde\phi_1 \overline{v}_1 \, d\Omega \\
    &\qquad + \int_\Omega D_2 \nabla \widetilde\phi_2 \cdot \nabla \overline{v}_2 
    + \Sigma_{a2}\, \widetilde\phi_2 \overline{v}_2 
    - \Sigma_{1 \rightarrow 2} \widetilde\phi_1 \overline{v}_2 \, d\Omega, \\
    &\qquad+\int_{\partial \O}  \alpha_1\widetilde\phi_1\overline{v}_1+\alpha_2\widetilde\phi_2\overline{v}_2\,dS \\
    b((f_1,f_2),(v_1,v_2))  &= \int_\Omega(\nu_1\Sigma_{f1} f_1+\nu_2\Sigma_{f2}f_2)\overline{v}_1d\Omega,
\end{align*}
i.e. the complex counterpart of the ones defining \eqref{eq:neutron-transport-weak}. We see immediately that these forms are bounded
    $$
    |a((\widetilde\phi_1,\widetilde\phi_2),(v_1,v_2))|  \leq M_a \|(\widetilde\phi_1,\widetilde\phi_2)\|_V\|(v_1,v_2)\|_V,
    $$
    where $M_a\coloneqq 2\max\{D_1,\Sigma_{a1}+\Sigma_{1\rightarrow 2},D_2,\Sigma_{1\rightarrow 2},C_1^2\alpha_1,C_2^2\alpha_2\}$ where $C_1$ and $C_2$ are constants that come from the trace inequality and 
$$
|b((f_1,f_2),(v_1,v_2)) | \leq M_b \|(f_1,f_2)\|_{Q}\|(v_1,v_2)\|_V,
$$
with $M_b \coloneqq \max\{\nu_1\Sigma_{f1}, \nu_2\Sigma_{f2}\}$. We highlight that, beyond sesquilinearity, this model is fundamentally non-symmetric as in block form it reads as
$$ \begin{bmatrix} \ten A_1 & \ten 0 \\ -\Sigma_{1\to 2}\ten I & \ten A_2 \end{bmatrix}\begin{bmatrix}\widetilde\phi_1 \\ \widetilde\phi_2 \end{bmatrix} = \begin{bmatrix} \nu_1 \Sigma_{f1} f_1 + \nu_2 \Sigma_{f2} f_2 \\ 0 \end{bmatrix}, $$
where $ \ten A_1 = -\nabla \cdot D_1 \nabla + (\Sigma_{a1} + \Sigma_{1\to 2})\ten I$ and $\ten A_2 = -\nabla \cdot D_2\nabla + \Sigma_{a2}\ten I$, with $\ten I$ the identity operator.  We start by noting that under mild conditions on the parameters (i.e. only positivity), problem \eqref{eq:source-problem} is well-posed.

\begin{lemma}\label{lemma:source-invertibility}
   Let $D_1,D_2,\Sigma_{a1},\Sigma_{a2},\Sigma_{1\to 2}$  be strictly positive constants. Then, there exists a unique solution $(\widetilde\phi_1,\widetilde\phi_2)$ in $V$ to \eqref{eq:source-problem} such that
        \begin{equation}\label{eq:phi1-apriori}
        \widetilde\alpha_1\|\widetilde\phi_1\|_1 \leq \left(\nu_1\Sigma_{f1}\|f_1\|_0+\nu_2\Sigma_{f2}\|f_2\|_0\right),
        \end{equation}
    for $\widetilde\alpha_1\coloneqq \min\{D_1,\Sigma_{a1} + \Sigma_{1\to 2}\}$, and
        \begin{equation}\label{eq:phi2-apriori} \widetilde\alpha_2\|\widetilde\phi_2\|_1 \leq \Sigma_{1\to 2}\|\widetilde\phi_1\|_0 \leq \frac{\Sigma_{1\to 2}}{\widetilde\alpha_1}\left(\nu_1\Sigma_{f1}\|f_1\|_0+\nu_2\Sigma_{f2}\|f_2\|_0\right),
        \end{equation}
    with $\widetilde\alpha_2 \coloneqq \min\{D_2, \Sigma_{a2}\}$. In particular it holds that
        \begin{equation}\label{eq:phi-apriori}
         \widetilde\alpha_1\|\widetilde\phi_1\|_1 + \widetilde\alpha_2\|\widetilde\phi_2\| \leq C_1\left(\nu_1\Sigma_{f1}\|f_1\|_0+\nu_2\Sigma_{f2}\|f_2\|_0\right),
         \end{equation}
        where $C_1\coloneqq \left(1 + \displaystyle\frac{\Sigma_{1\to 2}}{\widetilde\alpha_1}\right)$.
    \begin{proof}
        Setting $v_2=0$ in \eqref{eq:source-problem} yields the problem
        \begin{multline*}
            (D_1\nabla \widetilde\phi_1, \nabla v_1)_0 + (\Sigma_{a1}+\Sigma_{1\to 2})(\widetilde\phi_1,v_1)_0 \\= \nu_1\Sigma_{f1}(f_1, v_1)_0 + \nu_2\Sigma_{f2}(f_2,v_1)_0 \quad\forall v_1\in H^1(\O, \C),\end{multline*}
        which has a unique solution $\widetilde\phi_1$ in $H^1(\Omega, \C)$ in virtue of Lax-Milgram's lemma. We further obtain the \emph{a-priori} bound
        $$ D_1\|\nabla \widetilde\phi_1\|_0^2 + (\Sigma_{a1} + \Sigma_{1\to 2})\| \widetilde\phi_1\|_0^2 \leq \| f_1 \|_0\|\widetilde\phi_1\|_0,$$ which gives \eqref{eq:phi1-apriori}. Given the existence and uniqueness of $\widetilde\phi_1$, we now set $v_1=0$ in \eqref{eq:source-problem} to get the problem 
        $$ (D_2\nabla\widetilde\phi_2, \nabla v_2)_0 + \Sigma_{a2}(\widetilde\phi_2, v_2)_0 = \Sigma_{1\to 2}(\widetilde\phi_1, v_2) \qquad\forall v_2 \in H^1(\O, \C).$$
        Again, Lax-Milgram guarantees the existence and uniqueness of a $\widetilde\phi_2$ in $H^1(\O,\C)$ such that
        $$ D_2\|\nabla\widetilde\phi_2\|_0^2 + \Sigma_{a2}\|\widetilde\phi_2\|_0^2 \leq (\|f_2\|_0+ \|\widetilde\phi_1\|_0)\|\widetilde\phi_2\|_0,$$
        which yields \eqref{eq:phi2-apriori}. The combination of \eqref{eq:phi1-apriori} and \eqref{eq:phi2-apriori} gives \eqref{eq:phi-apriori}. This concludes the proof.
    \end{proof}
\end{lemma}
We note that the invertibility of the system poses no significant constraints on the parameters. This is not the case for a coupled ellipticity estimate to hold. We state this in the following lemma. 

\begin{lemma}\label{lmm:coercivity}
    If 
    \begin{equation}\label{eq:data-bounded}
        \Sigma_{1\to 2} \in \left(2\Sigma_{a1} - \sqrt{\Sigma_{a2}^2 + \Sigma_{a1}\Sigma_{a2}}, 2\Sigma_{a1} - \sqrt{\Sigma_{a2}^2 + \Sigma_{a1}\Sigma_{a2}}\right),
    \end{equation}
    then, bilinear form $a(\cdot,\cdot)$ involved  in problem \eqref{eq:source-problem} is elliptic.
    
    \begin{proof}
Let $(v_1,v_2)\in V$. From the definition of $a(\cdot,\cdot)$ we have 
\begin{align*}
    a((v_1,v_2),(v_1,v_2)) 
    &= \int_\Omega D_1 |\nabla v_1|^2 \,d\O
    + \int_{\O}(\Sigma_{a1} + \Sigma_{1 \rightarrow 2})\, v_1^2 \, d\Omega \\
    &\quad + \int_\Omega D_2 |\nabla v_2|^2\, d\O 
    + \int_{\O}\Sigma_{a2} v_2^2 \, d\Omega
    - \int_\Omega \Sigma_{1 \rightarrow 2}\, v_1 \overline{v}_2 \, d\Omega\\
    &\quad+\int_{\partial \O}  \alpha_1v_1^2+\,dS+\int_{\partial \O} \alpha_2v_2^2\,dS .
\end{align*}
Since $\displaystyle\int_{\partial \O}  \alpha_1v_1^2+\,dS+\displaystyle\int_{\partial \O} \alpha_2v_2^2\,dS>0$ and using the generalized Young inequality
$\displaystyle ab\leq\frac{a^2}{2\varepsilon}+\frac{\varepsilon b^2}{2}$ for all $\varepsilon>0$, we have
\begin{multline*}
    a((v_1,v_2),(v_1,v_2)) 
    \geq D_1 \|\nabla v_1\|_{0,\O}^2 
    + (\Sigma_{a1} + \Sigma_{1 \rightarrow 2})\|v_1\|_{0,\O}^2  \\
    \qquad + D_2 \|\nabla v_2\|_{0,\O}^2 
    + \Sigma_{a2} \|v_2\|_{0,\O}^2 
    - \Sigma_{1 \rightarrow 2}(v_1, v_2)_0\\
    \geq D_1 \|\nabla v_1\|_{0,\O}^2 
    + (\Sigma_{a1} + \Sigma_{1 \rightarrow 2})\|v_1\|_{0,\O}^2  \\
    \qquad + D_2 \|\nabla v_2\|_{0,\O}^2 
    + \Sigma_{a2} \|v_2\|_{0,\O}^2 -\Sigma_{1 \rightarrow 2}\|v_1\|_{0,\O}\|v_2\|_{0,\O}\\
    \geq D_1 \|\nabla v_1\|_{0,\O}^2 
    + (\Sigma_{a1} + \Sigma_{1 \rightarrow 2})\|v_1\|_{0,\O}^2  \\
    \qquad + D_2 \|\nabla v_2\|_{0,\O}^2 
    + \Sigma_{a2} \|v_2\|_{0,\O}^2+\Sigma_{1\rightarrow 2}\left(-\frac{\|v_1\|_{0,\O}^2}{2\varepsilon}-\frac{\varepsilon\|v_2\|_{0,\O}^2}{2} \right)\\
    =D_1\|\nabla v_1\|_{0,\O}^2+\underbrace{\left(\Sigma_{a1}+\Sigma_{1\rightarrow 2}-\frac{\Sigma_{1\rightarrow 2}}{2\varepsilon} \right)}_{C_1}\|v_1\|_{0,\O}^2\\
    \qquad+D_2\|\nabla v_2\|_{0,\O}^2+\underbrace{\left(\Sigma_{a2}-\frac{\Sigma_{1\rightarrow 2}\varepsilon}{2} \right)}_{C_2}\|v_2\|_{0,\O}^2.
\end{multline*}
Now, to ensure the positiveness of $C_1$, parameter $\varepsilon$ must be chosen as follows
\begin{equation}
\label{eq:epsilon1}
\displaystyle\frac{\Sigma_{1\rightarrow 2}}{2(\Sigma_{a1}+\Sigma_{1\rightarrow 2})}<\varepsilon,
\end{equation}
whereas for $C_2$ the parameters must satisfy
\begin{equation}
\label{eq:epsilon2}
\displaystyle \frac{2\Sigma_{a2}}{\Sigma_{1\rightarrow 2}}>\varepsilon.
\end{equation}
Hence, we require $\varepsilon$ to satisfy \eqref{eq:epsilon1} and \eqref{eq:epsilon2}. We note that this imposes a condition on the data as $\texttt{lhs} < \varepsilon < \texttt{rhs}$ implies $\texttt{lhs} < \texttt{rhs}$. After some elementary computations we obtain
    $$ \Sigma_{1\to 2}^2 - 4\Sigma_{a2}\Sigma_{1\to 2} - 4\Sigma_{a1}\Sigma_{a2} < 0,$$
which yields \eqref{eq:data-bounded}. Finally, we have 
\begin{equation*}
a((v_1,v_2),(v_1,v_2)) \geq \alpha \|(v_1,v_2)\|_V^2,
\end{equation*}
where the ellipticity constant $\alpha$ is defined by $\alpha\coloneqq\min\{D_1,D_2,C_1,C_2\}$. This concludes the proof.
\end{proof}

\end{lemma}

We highlight that possibly sharper estimates are computable by leveraging the Poincaré inequality and thus considering the diffusion coefficients as well, but we have not seen numerical evidence suggesting that the problem is sensitive to large variations of $\Sigma_{1\to 2}$. We believe this happens because all of the theory developed in this work does not require ellipticity to hold.

\section{The eigenvalue problem} \label{sec:eigen_problem}

Let us recall the eigenvalue problem of our interest: Find $\lambda$ in $\C$ and $(0,0)\neq(\phi_1,\phi_2)\in V$ such that
\begin{equation} \label{eq:eigen_problem}
    a((\phi_1,\phi_2),(v_1,v_2)) = \lambda b((\phi_1,\phi_2),(v_1,v_2))  \quad \forall (v_1,v_2) \in V,
\end{equation}
where $a$ and $b$ are the bilinear forms defined in Section~\ref{sec:source_problem}. As is customary in eigenvalue analysis, we need to introduce a solution operator. Let $T:V\to V$ be this operator, defined by the mapping 
\begin{equation*}
(f,g)\mapsto T(f,g)=(\widetilde{\phi}_1,\widetilde{\phi}_2),
\end{equation*}
where $(\widetilde{\phi}_1,\widetilde{\phi}_2)\in V$ is the solution of the source problem \eqref{eq:source-problem}. We consider the restriction of the operator to $V$ instead of using $Q$ to avoid unnecessary technicalities further ahead in the analysis. Lemma \ref{lemma:source-invertibility} implies that $T$ is well defined, and we additionally have the following regularity result (see \cite{MR775683,MR3340705}).
\begin{lemma}
\label{lmm:additional_regularity}
There exists $r_{\O}>1/2$ such that:
\begin{enumerate}
\item For all $f\in H^1(\O,\mathbb{C})$ and for all $\widehat{s}\in [1/2,r_{\O})$, the solution of \eqref{eq:source-problem} is such that 
$\widetilde{\phi}_1,\widetilde{\phi}_2\in H^{1+s}(\Omega,\mathbb{C})$ with $s\coloneqq\min\{\widehat{s},1\}$. Also, there exists a positive constant $C$ such that the following estimate holds 
\begin{equation*}
\label{eq:additional_reg_source_primal}
\|\widetilde{\phi}_1\|_{1+s,\O}+\|\widetilde{\phi}_2\|_{1+s,\O}\leq C\|(f,g)\|_V.
\end{equation*}
\item If $(w_1,w_2)\in V$ are eigenfunctions of \eqref{eq:eigen_problem} with eigenvalue $\lambda$, for all $r\in[1/2,r_{\O})$ there holds $w_1,w_2\in H^{1+r}(\O,\mathbb{C})$ and there exists a positive constant $\widehat{C}$ (depending on the eigenvalue) such that 
\begin{equation*}
\|w_1\|_{1+r,\O}+\|w_2\|_{1+r,\O}\leq \widehat{C}\|(w_1,w_2)\|_{V}.
\end{equation*}
\end{enumerate}
\end{lemma}

In virtue of Lemmas~\ref{lemma:source-invertibility} and \ref{lmm:additional_regularity}, we have the following important result.

\begin{corollary}\label{lmm:T-props}
    The solution operator $T$ is well-defined and compact.
    \begin{proof}
        The operator is well-defined in virtue of Lemma~\ref{lemma:source-invertibility}. The compactness is a consequence of the additional regularity from Lemma~\ref{lmm:additional_regularity} and the compact Sobolev embedding from $H^{1+s}(\O, \C)$ to $H^1(\O,\C)$.
    \end{proof}
\end{corollary}

\subsection{The adjoint problem}
Since the eigenvalue problem is not self-adjoint, the analysis requires to incorporate the adjoint eigenvalue problem. This problem is stated as follows: Find $\lambda^*$ in $\C$ and $(0,0)\neq (\phi_1^*,\phi_2^*)\in V$, such that 
\begin{equation}\label{eq:eigen_problem_dual}
    a((v_1,v_2),(\phi_1^*,\phi_2^*)) = \lambda^* b( (v_1,v_2),(\phi_1^*,\phi_2^*))  \quad \forall (v_1,v_2) \in V.
\end{equation}
For the analysis, we introduce the adjoint of $T$, which we denote by $T^*$, and is defined by 
\begin{equation*}
T^*:V\rightarrow V,\quad (f,g)\mapsto T^*(f,g)=(\widetilde{\phi}_1^*,\widetilde{\phi}_2^*),
\end{equation*}
where the pair $(\widetilde{\phi}_1^*,\widetilde{\phi}_2^*)$ is the unique solution of the adjoint source problem given by 
\begin{equation}\label{eq:source_problem_dual}
    a((v_1,v_2),(\widetilde{\phi}_1^*,\widetilde{\phi}_2^*)) = b( (v_1,v_2),(f,g))  \quad \forall (v_1,v_2) \in V.
\end{equation}
This problem has the following triangular structure:
$$ \begin{bmatrix} \ten A_1 & -\Sigma_{1\to 2}\ten I  \\ \ten 0 & \ten A_2 \end{bmatrix} \begin{bmatrix} \phi_1^* \\ \phi_2^* \end{bmatrix} = \begin{bmatrix}  \nu_1\Sigma_{f1} f_1 \\  \nu_2\Sigma_{f2} f_1 \end{bmatrix}. $$
Proceeding analogously to the primal source problem, we can obtain the following well-posedness result for the adjoint problem:
\begin{lemma}\label{lmm:adjoint_invertibility}
    There exists a unique solution $(\phi_1^*, \phi_2^*)$ in $V$ of problem \eqref{eq:source_problem_dual}. This solution satisfies
    \begin{equation*}\label{eq:phi2-dual-apriori}
        \alpha^*_2 \| \phi_2^* \|_1 \leq \nu_2 \Sigma_{f2} \| f_1\|_{0,\O},
    \end{equation*}
    where $\alpha^*_2 \coloneqq \min\{D_2, \Sigma_{a2}\}$ and
    \begin{equation*}\label{eq:phi1-dual-apriori}
        \alpha_1^* \| \phi_1^* \|_{1,\O} \leq \nu_1 \Sigma_{f1}\|f_1\|_{0,\O} + \Sigma_{1\to 2}\|\phi_2^*\|_{0,\O} \leq \left(\nu_1\Sigma_{f1} + \frac{\Sigma_{1\to 2}\nu_2\Sigma_{f2} }{\alpha_2} \right) \| f_1 \|_{0,\O},
    \end{equation*}
    with $\alpha_1^* \coloneqq \min\{D_1, \Sigma_{a1}+\Sigma_{1\to 2}\}$. In particular the following \emph{a-priori} estimate holds: 
    \begin{equation*}\label{eq:phi-dual-apriori}
        \alpha_1^*\| \phi_1^* \|_{1,\O} + \alpha_2^* \| \phi_2^* \|_{1,\O} \leq C_1^* \| f_1 \|_{0,\O},
    \end{equation*}
    where $C_1^* \coloneqq \left( \nu_1\Sigma_{f1} + \left(1 + \displaystyle\frac{\Sigma_{1\to 2}}{\alpha_2}\right)\nu_2\Sigma_{f2} \right)$.
\end{lemma}

Lemma~\ref{lmm:adjoint_invertibility} implies that $T^*$ is well defined. In addition, we also have from Lemma~\ref{lmm:additional_regularity} that there exists $s^*\geq 1/2$ such that
\begin{equation}
\label{eq:reg_dual_source}
\|\widetilde{\phi}_1^*\|_{1+s^*,\O}+\|\widetilde{\phi}_2^*\|_{1+s^*,\O}\leq C\|(f,g)\|_V,
\end{equation}
and for the eigenfunctions of \eqref{eq:eigen_problem_dual}, the following estimate holds 
\begin{equation*}
\|\phi_1^*\|_{1+r^*,\O}+\|\phi_2^*\|_{1+r^*,\O}\leq \widehat{C}\|(\phi_1^*,\phi_2^*)\|_{V},
\end{equation*}
where $C>0$ and $r^*\geq 1/2$. We thus have the adjoint analog of Lemma~\ref{lmm:T-props}.
\begin{corollary}\label{lmm:T-adjoint-props}
    The adjoint solution operator $T^*$ is well-defined and compact.
\end{corollary}
\begin{lemma}
    All eigenvalues of problem \eqref{eq:neutron-transport-weak-abstract} are real numbers.
    \end{lemma}
    \begin{proof}

From equation \eqref{eq:strong2} we define $\boldsymbol{M}\phi_2\coloneqq-\nabla(D_2\nabla \phi_2)+\Sigma_{a2}\phi_2$, 
so that
\begin{equation*}
     \boldsymbol{M}\phi_2 =  \Sigma_{1\to2}\phi_1.
\end{equation*}
Its associated bilineal form, for $v\in H^1(\O,\mathbb{C})$ is 
\begin{equation*}
    m(u,v)=\int_{\O}D_2\nabla u\cdot\nabla \overline{v}\,dx+\int_{\O}\Sigma_{a2}u\overline{v}\,dx.
\end{equation*}

We immediately observe that $m(u,v)$ is hermitian, since $m(u,v) = \overline{m(u,v)}$. Moreover, there exists a positive constant $\widetilde{M}\coloneqq D_2+\Sigma_{a2}$ such that $|m(u,v) |\leq \widetilde{M}\|u\|_{1,\O}\|v\|_{1,\O}$.
On the other hand, there exists $\alpha_m\coloneqq \min\{D_2,\Sigma_{a2}\}$ such that  $m(\cdot,\cdot)$ is coercive in $H^1(\O,\C)$, i.e. it holds that $m(u,u) \geq \alpha_m\|u\|^2_{1,\O}$.

Hence, the operator $\boldsymbol{M}:H^1(\O,\C)\to H^{1}(\O,\C)'$ defined by
\begin{equation*}
    \langle \boldsymbol{M}u,v\rangle \coloneqq m(u,v),
\end{equation*}
is invertible, so that $\phi_2$ can be written as $\phi_2 = \boldsymbol{M}^{-1}\phi_1$.
We also notice that $\boldsymbol{M}^{-1}$ is bounded, since
\begin{equation*}
    \|\boldsymbol{M}^{-1}\|_{H^{-1}}\leq  \frac{1}{ \min\{D_2,\Sigma_{a2}\}}.
\end{equation*}
Now it is possible to write \eqref{eq:strong1} by only using $\phi_1$, 
\begin{equation}
\label{eq:pos-vals}
-\nabla (D_1\nabla \phi_1)+(\Sigma_{a1}+\Sigma_{1\to2})\phi_1 = \frac{1}{k}\left(\nu_1\Sigma_{f1}\phi_1+\nu_2\Sigma_{f2}\boldsymbol{M}^{-1}\phi_1\right).
\end{equation}
Lastly, note that $\boldsymbol{M}^{-1}$ is positive definite. We define $\phi \coloneqq \boldsymbol{M}^{-1}u$, with $u\in L^2(\O,\C)$. Now
\begin{multline*}
\alpha_m\|\phi\|_{1,\O}^2\leq m(\phi,\phi)=\left\langle\boldsymbol{M}\phi,\phi\right\rangle_{H^1(\O,\C)',H^1(\O,\C)}\\
=\left\langle u,\boldsymbol{M}^{-1} u\right\rangle_{H^1(\O,\C)',H^1(\O,\C)}=\left\langle\widehat{\boldsymbol{M}
^{-1}u},u\right\rangle_{H^1(\O,\C)'',H^1(\O,\C)'},    
\end{multline*}
where $\widehat{\boldsymbol{M}
^{-1}u}$ is the canonical embedding   $H^1(\O,\C)\to H^1(\O,\C)''$. We note that these computations are required to establish the coercivity of $\boldsymbol{M}^{-1}$ in $H^1(\O)'$ as a left-acting operator. This operator, as shown in the previous computations, arises as a right-side operator, thus justifying the use of the canonical embedding. Indeed, the operator appearing in \eqref{eq:pos-vals} is the canonical embedding, but we preferred to avoid such heavy notation for simplicity and without loss of generality due to the Hilbert space setting.

Let us define the following lineal operators
\begin{equation*}
\begin{aligned}
    \mathcal{L}\phi_1 &\coloneqq -\nabla (D_1\nabla \phi_1)+(\Sigma_{a1}+\Sigma_{1\to2})\phi_1 ,\\
    \mathcal{R}\phi_1&\coloneqq(\nu_1\Sigma_{f1}\phi_1+\nu_2\Sigma_{f2}\boldsymbol{M}^{-1}\phi_1),
\end{aligned}
\end{equation*}
and their associated sesquilinear forms,
\begin{equation*}
\begin{aligned}
    \ell(u,v) &=\int_{\O} D_1\nabla u\cdot\nabla v+(\Sigma_{a1}+\Sigma_{1\to2})u\overline{v}\,dx ,\\
    r(u,v)&=\int_{\O}(\nu_1\Sigma_{f1}+\nu_2\Sigma_{f2}\boldsymbol{M}^{-1})u\overline{v}\,dx.
\end{aligned}
\end{equation*}


    Proceeding as we did for $m(\cdot,\cdot)$, it is possible to prove that forms $l(\cdot,\cdot)$ and $r(\cdot,\cdot)$ are bounded and hermitian. Moreover, $l(\cdot,\cdot)$ is $H^1(\O,\C)$ coercive and $r(\cdot,\cdot)$ positive semi-definite. 

Equation \eqref{eq:pos-vals} is equivalent to the problem,
\begin{equation*}
    \mathcal{L}\phi_1 = \lambda\mathcal{R}\phi_1,\quad \lambda \coloneqq\frac{1}{k}.
\end{equation*}
Where $\mathcal{L}$ and $\mathcal{R}$ are both self-adjoint, and positive definite. From Lemma \ref{lemma-apc} it follows that under these hypotheses, $\lambda\in \R^+$. 

    \end{proof}

\section{The finite element method}

\label{sec:fem}
The aim of this section is to introduce the discrete version of the eigenvalue problem under consideration. Let $\mathcal{T}_h$ be a shape regular family of meshes which subdivide the domain $\bar \Omega$ into  
triangles/tetrahedra that we denote by $K$. Let us denote by $h_K$
the diameter of any element $K\in\mathcal{T}_h$ and let $h$ be the maximum of the diameters of all the
elements of the mesh, i.e. $h\coloneqq \max_{K\in \cT_h} \{h_K\}$. We define for $k\geq 1$, the following finite dimensional space
\begin{equation*}
\widetilde{V}_h\coloneqq\{v\in H^1(\O,\mathbb{C})\,\,:\,\, v|_K\in\mathbb{P}_k(K)\},
\end{equation*}
where $\mathbb{P}_k(K)$ is the space of polynomials of degree $k$ defined in $K\in\CT_h$, and with it we define the discrete solution space $V_h\coloneqq\widetilde{V}_h\times \widetilde{V}_h$. 

Now we are in position to introduce the discrete counterpart of \eqref{eq:eigen_problem} as follows:  find $\lambda_h$ in $\C$ and $(\phi_{1,h},\phi_{2,h})$ in $V_h\setminus\{0\}$ such that
\begin{equation} \label{eq:eigen_problem_disc}
    a((\phi_{1,h},\phi_{2,h}),(v_{1,h},v_{2,h})) = \lambda_h b((\phi_{1,h},\phi_{2,h}),(v_{1,h},v_{2,h}))  \quad \forall (v_{1,h},v_{2,h}) \in V_h.
\end{equation}
Let us introduce the discrete solution operator $T_h$, defined by 
\begin{equation*}
T_h:V\rightarrow V_h,\quad (f,g)\mapsto T_h(f,g)=(\widetilde{\phi}_{1,h},\widetilde{\phi}_{2,h}),
\end{equation*}
where $(\widetilde{\phi}_{1,h},\widetilde{\phi}_{2,h})\in V_h$ is the solution of the following discrete source problem:
\begin{equation}
\label{eq:fuente_disc}
    a((\widetilde{\phi}_{1,h},\widetilde{\phi}_{2,h}),(v_{1,h},v_{2,h}))=b((f,g),(v_{1,h},v_{2,h}))\quad \forall (v_{1,h},v_{2,h})\in  V_h.
\end{equation}
We can reproduce Lemma~\ref{lemma:source-invertibility} verbatim to establish that there exists a unique discrete solution $(\widetilde{\phi}_{1,h},\widetilde{\phi}_{2,h})\in V_h$ for
\eqref{eq:fuente_disc}, implying that $T_h$ is well defined. On the other hand, the discrete adjoint eigenvalue problem associated to \eqref{eq:eigen_problem_dual} reads as follows: find $\lambda_h^*$ in $\C$ and $(\phi_{1,h}^*,\phi_{2,h}^*)$ in $V_h\setminus\{0\}$ such that 
\begin{equation*} \label{eq:eigen_problem_dual_disc}
    a((v_{1,h},v_{2,h}),(\phi_{1,h}^*,\phi_{2,h}^*)) = \lambda^* b( (v_{1,h},v_{2,h}),(\phi_{1,h}^*,\phi_{2,h}^*))  \quad \forall (v_{1,h},v_{2,h}) \in V_h.
\end{equation*}
We also introduce the adjoint discrete solution operator $T_h^*$ defined by 
\begin{equation*}
T_h^*:V\rightarrow V_h,\quad (f,g)\mapsto T_h^*(f,g)=(\widetilde{\phi}_{1,h}^*,\widetilde{\phi}_{2,h}^*),
\end{equation*}
where $(\widetilde{\phi}_{1,h}^*,\widetilde{\phi}_{2,h}^*)\in V_h$ is the solution of the following source problem 
\begin{equation*}
\label{eq:fuente_dual_disc}
    a((v_{1,h},v_{2,h}),(\widetilde{\phi}^*_{1,h},\widetilde{\phi}^*_{2,h}))=b((v_{1,h},v_{2,h}),(f,g))\quad \forall (v_{1,h},v_{2,h})\in  V_h,
\end{equation*}
which is well-posed in virtue of our previous developments, thus $T_h^*$ is well defined.

\subsection{Convergence}
We begin with some definitions. Let $T: X\rightarrow X$ be a compact operator on a complex Hilbert space $X$. The resolvent operator associated to $T$ is the set of all the complex numbers $z\in\C$ for which $(zI-T)$ is invertible, i.e., 
\begin{equation*}
\rho(T)\coloneqq\{z\in\mathbb{C}\,:\, (zI-T)^{-1}\,\,\text{exists and is bounded}\}. 
\end{equation*}
Hence, the spectrum of $T$
is defined by $\sp(T)\coloneqq\mathbb{C}\setminus\rho(T)$. Since $T$ is compact, if $\mu\in\sp(T)$, then $\mu$ is an isolated  eigenvalue of $T$ and its generalized eigenspace  is finite dimensional. On the other hand, if $\mu$ is a non-vanishing eigenvalue of $T$ there exists a smallest positive integer $\alpha$ called the ascent of $(\mu I-T)$ such that the following relation holds
\begin{equation*}
\mathcal{N}((\mu I-T)^{\eta})=\mathcal{N}((\mu I-T)^{\eta+1}),
\end{equation*} 
where $\mathcal{N}$ denotes the null space. We denote by $m$  the algebraic multiplicity of $\mu$  which we define by $m\coloneqq\dim\mathcal{N}((\mu I-T)^{\eta})$ whereas the geometric multiplicity of $\mu$ is  $\dim\mathcal{N}(\mu\boldsymbol{I}-\boldsymbol{T})$. For simplicity, we consider the non-defective non-self-adjoint eigenvalues in this paper (i.e., the ascent $\eta=1$) , allowing that  the generalized eigenspace is the same as the eigenspace.

Now that we have our continuous and discrete formulations, the next step is to analyze convergence. More precisely, we are interested in the convergence of $T_h$ to $T$ as $h\rightarrow 0$ (and for their adjoint counterparts). If this convergence between operators is established, the theory of  \cite{MR1115235} can be applied due to the compactness of $T$. We prove the following result.

\begin{lemma}
\label{lmm:conv1} 
Let $(f,g)\in V$. Then, the following estimate holds
\begin{equation*}
\|(T-T_h)(f,g)\|_{V}\leq Ch^{\min\{k,s\}}\|(f,g)\|_Q,
\end{equation*}
where $s>0$ is the regularity exponent given in Lemma \ref{lmm:additional_regularity} and $C>0$ is independent of $h$.
\end{lemma}
\begin{proof}
According to the definitions of $T$ and $T_h$, for $(f,g)\in V$ we have that $(\widetilde{\phi}_1,\widetilde{\phi}_2)\coloneqq T(f,g)$ and  $(\widetilde{\phi}_{1,h},\widetilde{\phi}_{2,h})\coloneqq T_h(f,g)$. Proceeding  as in Lemma~\ref{lemma:source-invertibility}, we can prove a Ceá estimate without the ellipticity of $a$. We first consider $v_2 = 0$ in \eqref{eq:eigen_problem}, $v_{2,h}=0$ in \eqref{eq:eigen_problem_disc}, and set the error $e_{1,h} = \widetilde\phi_1 - \widetilde\phi_{1.h}$ to obtain a Galerkin orthogonality as
    $$ a_1(e_h, v_{1,h}) \coloneqq D_1(\nabla e_h, \nabla v_{1,h})_0 + (\Sigma_{a1} + \Sigma_{1\to 2})(e_h, v_{1,h})_0 = 0 \qquad\forall v_{1,h} \in \widetilde V_h.$$
We note again that $a_1(\cdot,\cdot)$ is elliptic with constant by $\alpha_1 \coloneqq \min\{D_1, \Sigma_{a1} + \Sigma_{1\to 2}\}$, and continuous with constant bounded by  $M_{1}\coloneqq \max\{D_1, \Sigma_{a1}+ \Sigma_{1\to 2}\}$. Everything together yields a Ceá estimate for $\widetilde\phi_1$, given by
\begin{equation}\label{eq:a1-convergence}
    \| e_{1,h} \|_1 \leq \frac{M_1}{\alpha_1} \text{dist}(\widetilde\phi_1, \widetilde V_h). 
\end{equation}
Setting $v_1=0$ in \eqref{eq:eigen_problem}, $v_{1,h}=0$ in \eqref{eq:eigen_problem_disc}, and $e_{2,h}= \widetilde\phi_2 - \widetilde\phi_{2,h}$, their difference yields the error equation
    $$ a_2(e_{2,h}, v_{2,h})\coloneqq D_2(\nabla e_{2,h}, \nabla v_{2,h})_0 + \Sigma_{a2}(e_{2,h}, v_{2,h})_0 = \Sigma_{1\to 2}(e_{1,h}, v_{2,h})_0 \quad \forall v_{2,h} \in \widetilde V_h.$$
Here, $a_2$ is elliptic with constant bounded by $\alpha_1\coloneqq \min\{D_2, \Sigma_{a2}\}$, and continuous with constant bounded by $M_2\coloneqq \max\{D_2, \Sigma_{a2}\}$. This gives the Ceá estimate 
    \begin{equation}\label{eq:a2-convergence}
        \| e_{2,h} \|_1 \leq \frac{M_2}{\alpha_2} \text{dist}(\widetilde\phi_2, \widetilde V_h).
    \end{equation}
We have thus established that a global Ceá estimate holds simply by adding \eqref{eq:a1-convergence} and \eqref{eq:a2-convergence}. Putting everything together we obtain the desired result: 
\begin{multline}
\label{eq:cea}
\|(T-T_h)(f,g)\|_V=\|(\widetilde{\phi}_1,\widetilde{\phi}_2)-(\widetilde{\phi}_{1,h},\widetilde{\phi}_{2,h})\|_V\\
\leq \widetilde C \inf_{(\widetilde{v}_{1,h},\widetilde{v}_{2,h})\in V_h}\|(\widetilde{\phi}_1,\widetilde{\phi}_2)-(\widetilde{v}_{1,h},\widetilde{v}_{2,h})\|_V,
\end{multline}
where $\widetilde C\coloneqq \max\{\frac{M_1}{\alpha_1}, \frac{M_2}{\alpha_2}\}$. Now, let $\mathcal{L}_h$ be the classic Lagrange interpolation operator (see for instance  \cite{MR2050138}). Hence, $\mathcal{L}_h\widetilde{\phi}_1$ and $\mathcal{L}_h\widetilde{\phi}_2$ belongs to $\tilde{V}_h$. Now, invoking the additional regularity $\widetilde{\phi}_1,\widetilde{\phi}_2\in H^{1+s}(\O)$, and using the approximation properties of  $\mathcal{L}_h$ on \eqref{eq:cea},  we have
\begin{multline*}
\|(T-T_h)(f,g)\|_V\leq \|\widetilde{\phi}_1-\mathcal{L}_h\widetilde{\phi}_1\|_{1,\O}+\|\widetilde{\phi}_2-\mathcal{L}_h\widetilde{\phi}_2\|_{1,\O}\\
\leq Ch^{\min\{k,s\}}(\|\widetilde{\phi}_1\|_{1+s,\O}+\|\widetilde{\phi}_2\|_{1+s,\O})\leq Ch^{\min\{k,s\}}\|(f,g)\|_{Q},
\end{multline*}
where the constant $C>0$ is independent of $h$. This concludes the proof.
\end{proof}

For the adjoint counterparts the previous result is also possible to be obtained under the same arguments. For simplicity we skip the details. 
\begin{lemma}
\label{lmm:conv2_dual} 
Let $(f,g)\in V$. Then, the following estimate holds
\begin{equation*}
\|(T^*-T^*_h)(f,g)\|_{V}\leq Ch^{\min\{k,s^*\}}\|(f,g)\|_Q,
\end{equation*}
where $s^*>0$ is the regularity exponent given in \eqref{eq:reg_dual_source} and $C>0$ is independent of $h$.
\end{lemma}

 Now, with Lemmas \ref{lmm:conv1} and \ref{lmm:conv2_dual}  at hand, together with  the results of \cite[Chapter IV]{MR0203473} and \cite[Theorem 9.1]{MR2652780}, we conclude that our numerical method does not introduce spurious eigenvalues.
\begin{theorem}
	\label{thm:spurious_free}
	Let $K\subset\mathbb{C}$ be any compact set contained in $\sp(T)$. Then, there exists $h_0>0$ such that $K\subset \sp(T_h)$ for all $h<h_0$.
\end{theorem}

\subsection{Error estimates}
\label{sec:conv}
Now our aim is to obtain  a priori error estimates for the eigenfunctions and eigenvalues. To do this, we first recall some definitions.

 Let $\mu$ be a nonzero isolated eigenvalue of $T$ with algebraic multiplicity $m$, i.e, $\mu$=$\mu_k=\mu_{k+1}=...=\mu_{k+m-1}$  and let $\Gamma$
be a simple closed curve of the complex plane lying in $\rho(T)$, which contains an eigenvalue $\mu$ and no other eigenvalues. The spectral projections of $E$ and $E^*$, associated to $T$ and $T^*$ and their respective eigenvalues $\mu$ and $\mu^*$, are defined, respectively, in the following way:
\begin{enumerate}
\item[a)] The spectral projector of $T$ associated to $\mu$ is $\displaystyle E\coloneqq\frac{1}{2\pi i}\int_{\Gamma} (z I-T)^{-1}\,dz.$
\item[b)] The spectral projector of $T^*$ associated to $\bar{\mu}$ is $\displaystyle E^*\coloneqq\frac{1}{2\pi i}\int_{\Gamma} (\bar{z}I-T^*)^{-1}\,d\bar{z},$
\end{enumerate}
where $I$ represents the identity operator and  $E$ and $E^*$ are the projections onto the generalized eigenspaces $R(E)$ and $R(E^*)$, respectively. Then $R(E)=\mathcal{N}((\mu I-T)^{\alpha})$ represents the generalized eigenspace related to $\mu$ and the operator $T$. Let us recall theta in our analysis $\alpha=1$.  For a more detailed discussion about these definitions, we refer to \cite[Section 6]{MR2652780} and \cite[Chapter II]{BO}.

The convergence in norm stated in Lemma \ref{lmm:conv1} gives as consequence  the  existence of  $m$ eigenvalues lying  in $\Gamma$, which we denote by  $\mu_h^{(1)},\ldots,\mu_h^{(m)}$, repeated according their respective multiplicities, that converge to $\mu$ as $h$ goes to zero. This motivates the  definition of the following discrete  spectral projection
$$
E_h\coloneqq\frac{1}{2\pi i}\int_{\Gamma} (\mu I-T_h)^{-1}\,d\mu,
$$
which is precisely a projection onto the discrete invariant subspace $R(E_h)$ of $T$, spanned by the generalized eigenvector of $T_h$ corresponding to 
 $\mu_h^{(1)},\ldots,\mu_h^{(m)}$.

Another necessary ingredient  for the error analysis is the \textit{gap} $\hdel(\cdot,\cdot)$ between two closed
subspaces $\mathfrak{X}$ and $\mathfrak{Y}$ of $L^2(\O,\mathbb{C})$, which is defined by 
$$
\hdel(\mathfrak{X},\mathfrak{Y})
\coloneqq\max\big\{\delta(\mathfrak{X},\mathfrak{Y}),\delta(\mathfrak{Y},\mathfrak{X})\big\}, \text{ where } \delta(\mathfrak{X},\mathfrak{Y})
\coloneqq\sup_{\underset{\left\|x\right\|_{0,\O}=1}{x\in\mathfrak{X}}}
\left(\inf_{y\in\mathfrak{Y}}\left\|x-\boldsymbol{y}\right\|_{0,\O}\right).
$$

Let us assume that there exist $r,r^*>0$ such that $R(E)\subset H^r(\O,\mathbb{C})$ and $R(E^*)\subset H^{r^*}(\O,\mathbb{C})$. Now we present the main result of this section.
\begin{theorem}
\label{thm:errors1}
Let $\mu$ be an eigenvalue of $T$ with algebraic multiplicity  $m$ and $\mu_h^{(k)}$, $k=1,...,m$ be the $m$ eigenvalues of $T_h$ approximating $\mu$. Then, for $k\geq 1$  the following estimates hold
$$
\hdel(R(E),R(E_h))\leq Ch^{\min\{k,r\}}\quad\text{and}\quad
|\mu-\widehat{\mu}_h|\leq  Ch^{\min\{k,r\}+\min\{k,r^*\}},
$$
where $\widehat{\mu}_h=\displaystyle\dfrac{1}{m}\sum_{k=1}^m\mu_h^{(m)}$ and $C>0$ is independent of $h$.
\end{theorem}
\begin{proof}
The gap between the eigenspaces is a direct consequence of the Lemma \ref{lmm:conv1}. 
For  the double order of convergence for the eigenvalues we procede as follows: let $\{(\phi_{1,\ell},\phi_{2,\ell})\}_{\ell=1}^m$ be such that $T (\phi_{1,\ell},\phi_{2,\ell})=\mu (\phi_{1,\ell},\phi_{2,\ell})$, for $\ell=1,\ldots,m$. An adjoint basis for $R(E^*)$ is $\{(\phi_{1,\ell}^*,\phi_{2,\ell}^*)\}_{\ell=1}^m$ that satisfies $a((\phi_{1,\ell},\phi_{2,\ell}),(\phi_{1,\ell}^*,\phi_{2,\ell}^*))=\delta_{\ell,l},$
where $\delta_{\ell,l}$ represents the Kronecker delta.
On the other hand, the following identity holds
\begin{multline*}
|\mu-\widehat{\mu}_h|\leq  \frac{1}{m}\sum_{\ell=1}^m|\langle(T-T_h)(\phi_{1,\ell},\phi_{2,\ell}),(\phi_{1,\ell}^*,\phi_{2,\ell}^*) \rangle|\\
+\|(T-T_h)|_{R(E)} \|_{V} \|(T^*-T_h^*)|_{R(E^*)}\|_{V},
\end{multline*}
where $\langle\cdot,\cdot\rangle$ denotes the corresponding duality pairing. Observe that from Lemmas \ref{lmm:conv1}  and \ref{lmm:conv2_dual} , together with the regularity of the generalized spaces $R(E)$ and $R(E^*)$, we have 
\begin{equation}
\label{eq:doble_1}
\|(T-T_h)|_{R(E)} \|_{V} \|(T^*-T_h^*)|_{R(E^*)}\|_{V}\leq Ch^{\min\{k,r\}+\min\{k,r^*\}}.
\end{equation}
On the other hand,  for the first term on the right-hand side we note that
\begin{multline}
\label{eq:doble_2}
|\langle(T-T_h)(\phi_{1,\ell},\phi_{2,\ell}),(\phi_{1,\ell}^*,\phi_{2,\ell}^*)\rangle|\leq C|a((T-T_h)(\phi_{1,\ell},\phi_{2,\ell}),(\phi_{1,\ell}^*,\phi_{2,\ell}^*))|\\
= C\inf_{(v_{1,\ell,h}^*,v_{2,\ell,h}^*)\in V_h} |a((T-T_h)(\phi_{1,\ell},\phi_{2,\ell}),(\phi_{1,\ell}^*,\phi_{2,\ell}^*)-(v_{1,\ell,h}^*,v_{2,\ell,h}^*))|\\
\leq C |a((T-T_h)(\phi_{1,\ell},\phi_{2,\ell}),(\phi_{1,\ell}^*,\phi_{2,\ell}^*)-(\phi_{1,\ell,h}^*,\phi_{2,\ell,h}^*))|\\
\leq C \|(T-T_h)(\phi_{1,\ell},\phi_{2,\ell})\|_V\|(\phi_{1,\ell}^*,\phi_{2,\ell}^*)-(\phi_{1,\ell,h}^*,\phi_{2,\ell,h}^*)\|_V\\
\leq C h^{\min\{k,r\}+\min\{k,r^*\}}\|(\phi_{1,\ell},\phi_{2,\ell})\|_V\|(\phi_{1,\ell}^*,\phi_{2,\ell}^*)\|_V.
\end{multline}
Hence, combining \eqref{eq:doble_1} and \eqref{eq:doble_2}, we conclude the proof.
\end{proof}

\section{Numerical tests}
\label{sec:numerics}
In this section we conduct several numerical tests to assess the performance of our scheme across different geometries and physical configurations. These are implemented using the DOLFINx software \cite{barrata2023dolfinx,scroggs2022basix}, where the SLEPc eigensolver \cite{hernandez2005slepc} and the MUMPS linear solver are employed to solve the resulting generalized eigenvalue problem. The convergence rates for each eigenvalue are determined using least-square fitting and highly refined meshes.  More precisely, if $\lambda_h$ is a discrete complex eigenvalue, then the rate of convergence $\eta$ is calculated by extrapolation and the least square  fitting
$$
\lambda_{h}\approx \lambda_{\text{extr}} + Ch^{\eta},
$$
where $\lambda_{\text{extr}}$ is the extrapolated eigenvalue given by the fitting. 

In what follows, we denote the mesh resolution by $N$, which is connected to the mesh-size $h$ through the relation $h\sim N^{-1}$. We also denote the number of degrees of freedom by $\texttt{dof}$. The relation between $\texttt{dof}$ and the mesh size is given by $h\sim\texttt{dof}^{-1/n}$, with $n\in\{2,3\}$. 

In the forthcoming experiments, unless otherwise stated, the values for the diffusion  and cross-sections parameters are positive and constant, as we show in Table~\ref{tab:parameters}. 
\begin{table}[H]
    \centering
    \begin{tabular}{c|c}\toprule
 parameter&value\\\midrule
        $D_1$ & $1.0$\\
        $D_2$ & $0.5$\\
        $\Sigma_{a1}$ & $0.2$\\
        $\Sigma_{a2}$ & $0.1$\\
        $\nu_1\Sigma_{f1}$ & $0.3$\\
        $\nu_2\Sigma_{f2}$ & $0.1$\\
        $S_{12}$ & $0.1$\\ \bottomrule
    \end{tabular}
    \caption{Physical and geometrical parameters for the experiments}
    \label{tab:parameters}
\end{table}

Moreover, our interest is to explore the proposed method in different contexts. To do this, we focus our attention on the following scenarios:
    \begin{itemize}
        \item Convergence tests: at this point, our aim is to computationally validate the theoretical results, more precisely, the double order of convergence given by Theorem \ref{thm:errors1}. To do this, we will compute the spectrum on the following domains:
            \begin{itemize}
                \item A unit square domain.
                \item An L-shaped domain.
                \item A circular domain.
                \item A unit cube.
            \end{itemize}
        \item A realistic simulation on the standard IAEA benchmark \cite{iaea}.
    \end{itemize}

\subsection{Unit square with Dirichlet boundary conditions}

In this test we consider the unit square $\O\coloneqq(0,1)^2$ as computational domain with homogeneous Dirichlet boundary conditions on $\partial\O$. We expect for this geometrical configuration sufficiently regular eigenfunctions and hence, optimal orders of convergence for the method. In  Table~\ref{tab:eigen_convergence_vertical_1} we report the first five computed eigenvalues  for different meshes and different polynomial degrees. In the column "Order" we present the order of convergence obtained with our least square fitting whereas the column $\lambda_{ext}$ reports the extrapolated values.

\begin{table}[H]
    \centering
    \renewcommand{\arraystretch}{1.2}
    \setlength{\tabcolsep}{5pt}
    \begin{tabular}{c c c c c c c c}
        \toprule
        $k$ &  & $N=8$ & $N=16$ & $N=32$ & $N=64$ & Order & $\lambda_{ext}$ \\
        \midrule
        \multirow{5}{*}{1}
        &  & 69.1292 & 67.2100 & 66.7333 & 66.6144 & 2.01 & 66.5757\\
        &  & 176.2102 & 167.9989 & 165.9527 & 165.4414 & 2.00 & 165.2679\\
        &  & 182.7912 & 169.5538 & 166.3355 & 165.5367 & 2.04 & 165.2945 \\
        &  & 302.8717 & 274.0154 & 266.4978 & 264.6009 & 1.95 & 263.9141\\
        &  & 380.7322 & 342.3123 & 332.8873 & 330.5443 & 2.02 & 329.7751\\
        \midrule
        \multirow{5}{*}{2}
        &  & 66.5895 & 66.5757 & 66.5748 & 66.5747 & 3.97 & 66.5747 \\
        &  & 165.4041 & 165.2797 & 165.2715 & 165.2710 & 3.93 & 165.2710 \\
        &  & 165.5162 & 165.2870 & 165.2720 & 165.2710 & 3.94 & 165.2710\\
        &  & 264.8394 & 264.0262 & 263.9709 & 263.9673 & 3.88 & 263.9669\\
        &  & 331.0075 & 329.8483 & 329.7698 & 329.7648 & 3.89 & 329.7644\\
        \midrule
        \multirow{5}{*}{3}
         &&$66.5748065$  & $66.5747707$  & $66.574770$  & $66.5747701$  & 6.03 & 66.5747 \\
        && $165.271954$ & $165.271049$ & $165.271035$ & $165.271035$ & 5.98 & 165.2710 \\
        && $165.272449$ & $165.271057$ & $165.271035$ & $165.271035$ & 5.99 &  165.2710\\
        && $263.976213$ & $263.967279$ & $263.967137$ & $263.967135$ & 5.98 & 263.9671 \\
        && $329.779588$ & $329.764758$ & $329.764520$ & $329.764516$ & 5.96 & 329.7645 \\
        \bottomrule
    \end{tabular}
    \caption{Convergence behavior of the first five lowest computed eigenvalues on the unit square domain, with homogeneous Dirichlet boundary conditions and polynomial degrees $k=1,2,3$.}
    \label{tab:eigen_convergence_vertical_1}
\end{table}

We observe in Table \ref{tab:eigen_convergence_vertical_1} that the optimal order of convergence is attained with our method for finite element spaces of orders 1, 2, and 3. In fact, the quadratic order of convergence is clear when the spectrum is approximated. Moreover, all the eigenvalues that we have obtained are real, despite to the non-symmetry of the eigenvalue problem. Also in Figure \ref{fig:four_images_square} we present plots of the first four approximated eigenfunctions. We stress that since the computed eigenvalues are real, the plots represent precisely real eigenfunctions whith imaginary parts equal to zero.

\begin{figure}[h!]
    \centering
    $\phi_1$\\
\includegraphics[width=0.22\textwidth]{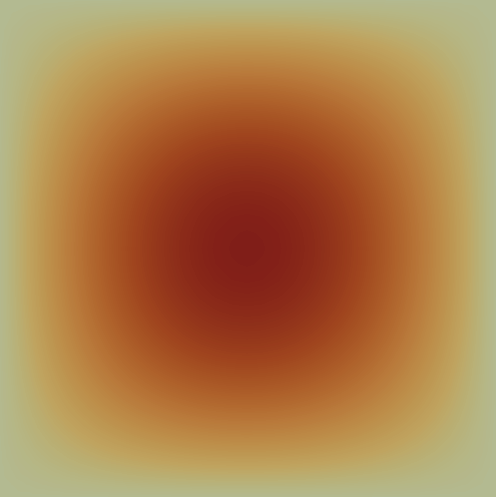}
    \includegraphics[width=0.22\textwidth]{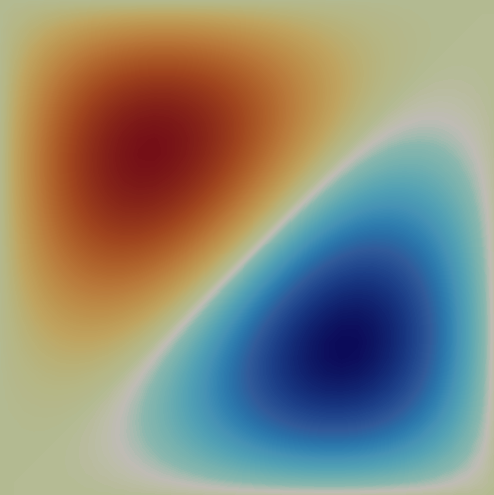}
    \includegraphics[width=0.22\textwidth]{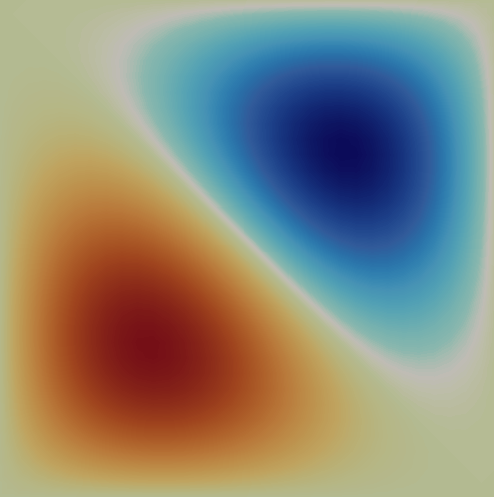}
    \includegraphics[width=0.22\textwidth]{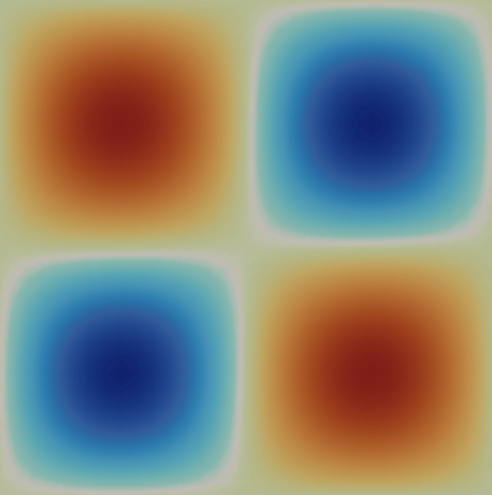}
    \includegraphics[width=0.065\textwidth]{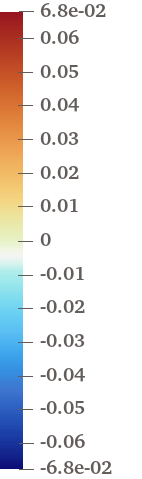}
    $\phi_2$\\
    \includegraphics[width=0.22\textwidth]{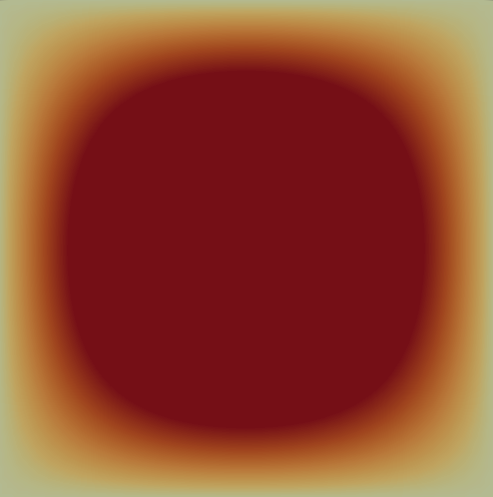}
    \includegraphics[width=0.22\textwidth]{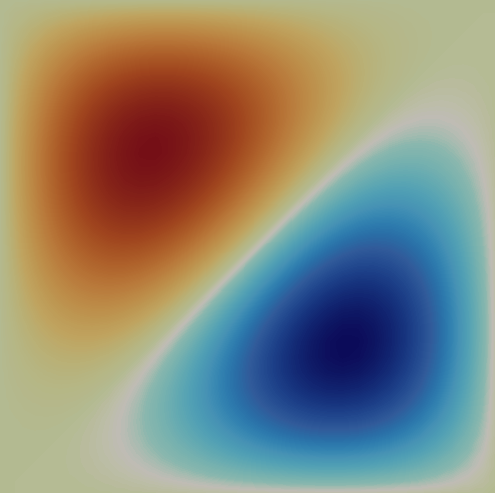}
    \includegraphics[width=0.22\textwidth]{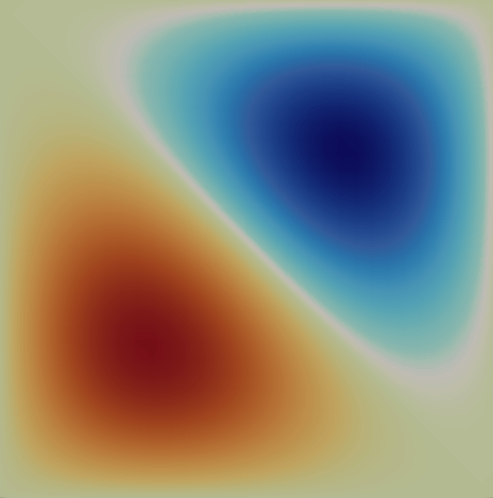}
    \includegraphics[width=0.22\textwidth]{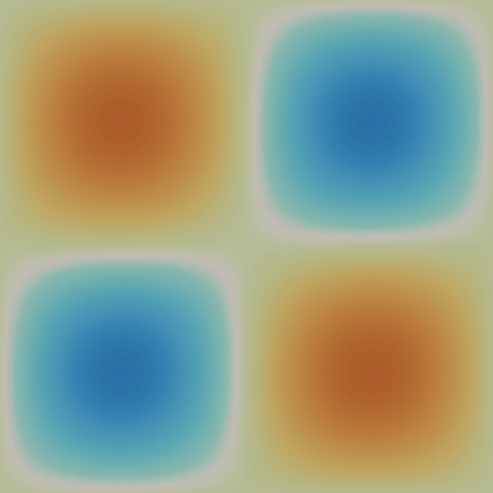}
    \includegraphics[width=0.07\textwidth]{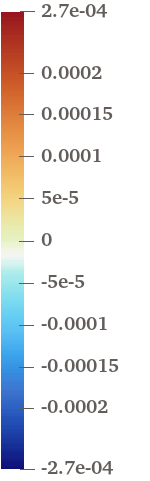}
  
    \caption{Eigenfunctions associated to the first four eigenvalues on the unit square domain with $k=1$. }
    \label{fig:four_images_square}
\end{figure}

\subsection{L-shaped with Dirichlet boundary conditions}
For this test we consider the non-convex domain commonly refered to as the L-shaped domain. The re-entrant angle on this geometry leads to non-sufficiently smooth eigenfunctions that affect the order of convergence of the associated eigenvalues. Once again, we explore the method for this geometry considering $k=1,2,3$ as polynomial approximations and different refinement levels in order to compute the order of convergence. 
\begin{table}[H]
    \centering
    \renewcommand{\arraystretch}{1.2}
    \setlength{\tabcolsep}{5pt}
    \begin{tabular}{c c c c c c c c}
        \toprule
        $k$ &  & $N=8$ & $N=16$ & $N=32$ & $N=64$ & Order & $\lambda_{ext}$ \\
        \midrule
        \multirow{5}{*}{1}
        &  & 137.1459 & 131.8404 & 130.1638 & 129.6033 & 1.65 & 129.3614\\
        &  & 213.7216 & 206.2422 & 204.1467 & 203.5943 & 1.85 & 203.3637\\
        &  & 283.0006 & 268.8414 & 265.2298 & 264.2860 & 1.97 & 263.9817\\
        &  & 439.0850 & 405.3242 & 397.2384 & 395.1080 & 2.04 & 394.5212\\
        &  & 486.0530 & 442.4305 & 430.9624 & 427.6495 & 1.90 & 426.5684\\
        \midrule
        \multirow{5}{*}{2}
        &  & 129.8972 &129.5474 & 129.4085 & 129.3477 & 1.30 & 129.3096\\
        &  & 203.5787 & 203.4233 & 203.4094 &203.4079 & 3.46 & 203.4079\\
        &  & 264.2352 & 263.9849 & 263.9683 &263.9672 & 3.91 & 263.9671\\
        &  & 395.2765 & 394.4597 & 394.4018 &394.3977 & 2.82 & 394.3974\\
        &  & 428.5742 & 426.9219 & 426.5294 & 426.3784 & 1.95 & 426.3566\\
        \midrule
        \multirow{5}{*}{3}
        && 129.5394 & 129.4028 & 129.3475 & 129.3233 & 1.28 & 129.3076 \\
        && 203.4137 & 203.4086 & 203.4078 & 203.4077 & 2.65 & 203.4076 \\
        && 263.9685 & 263.9671 & 263.9671 & 263.9671 & 5.80 & 263.9671 \\
        && 394.4082 & 394.3979 & 394.3975 & 394.3974 & 4.41 & 394.3974 \\
        && 426.8356 & 426.5116 & 426.3778 & 426.3188 & 1.25 & 426.2780 \\
        \bottomrule
    \end{tabular}
    \caption{Convergence behavior of the first five lowest computed eigenvalues on the L shaped domain, with homogeneous Dirichlet boundary conditions and polynomial degrees $k=1,2,3$.}
    \label{tab:eigen_convergence_L_shaped}
\end{table}

From Table \ref{tab:eigen_convergence_L_shaped} we observe that, as expected, the order of convergence of some eigenvalues is deteriorated due to the singularity of the associated eigenfunctions. This lack of optimal order is independent of the degree on the polynomial approximation.

\subsection{Circular domain with Dirichlet boundary conditions}

In this section we show that the method is capable of approximating the eigenvalues on a domain that goes beyond the developed theory. More specifically, we consider a circular domain defined by $\O\coloneqq\{(x,y)\in\mathbb{R}^2\,\,: x^2+y^2<1\}$ with homogeneous Dirichlet boundary conditions. For this domain, all the eigenfunctions are smooth but, the approximation of a curved domain with triangles will lead to a loss on the convergence order of the method when $k>1$. We report our results on Table \ref{tab:eigen_convergence_circle}.
\begin{table}[H]
    \centering
    \renewcommand{\arraystretch}{1.2}
    \setlength{\tabcolsep}{5pt}
    \begin{tabular}{c c c c c c c c}
        \toprule
        $k$ &  & $N=8$ & $N=16$ & $N=32$ & $N=64$ & Order & $\lambda_{ext}$ \\
        \midrule
        \multirow{5}{*}{1}
        &  & 20.1587 & 20.0807 & 20.0605 & 20.0555 & 1.96 & 20.0536\\
        &  & 50.3899 & 49.8901 & 49.7605 & 49.7280 & 1.96 & 49.7165\\
        &  & 50.3956 & 49.8910 & 49.7605 & 49.7280 & 1.96 & 49.7162\\
        &  & 90.8556 & 89.2505 & 88.8327 & 88.7280 & 1.95 & 88.6893\\
        &  & 90.8715 & 89.2535 & 88.8328 & 88.7280 & 1.95 & 88.6885\\
        \midrule
        \multirow{5}{*}{2}
        &  & 20.1043 & 20.0664 & 20.0569 & 20.0546 & 2.00 & 20.0538\\
        &  & 49.8468 & 49.7494 & 49.7251 & 49.7191 & 2.01 & 49.7171\\
        &  & 49.8471 & 49.7494 & 49.7251 & 49.7191 & 2.01 & 49.7171\\
        &  & 88.9324 & 88.7512 & 88.7072 & 88.6964 & 2.04 & 88.6930\\
        &  & 88.9331 & 88.7512 & 88.7072 & 88.6964 & 2.05 & 88.6931\\
        \midrule
        \multirow{5}{*}{3}
        && 20.1035 & 20.0663 & 20.0569 & 20.0546 & 1.98 & 20.0537 \\
        && 49.8433 & 49.7490 & 49.7251 & 49.7191 & 1.98 & 49.7170 \\
        && 49.8434 & 49.7490 & 49.7251 & 49.7191 & 1.98 & 49.7170 \\
        && 88.9196 & 88.7501 & 88.7071 & 88.6964 & 1.98 & 88.6926 \\
        && 88.9196 & 88.7501 & 88.7071 & 88.6964 & 1.98 & 88.6926 \\
        \bottomrule
    \end{tabular}
    \caption{Convergence behavior of the first five lowest computed eigenvalues on the unit circle shaped domain, with homogeneous Dirichlet boundary conditions and polynomial degrees $k=1,2,3$.}
    \label{tab:eigen_convergence_circle}
\end{table}
Table~\ref{tab:eigen_convergence_circle} shows that the method  works as expected. More precisely, all the eigenvalues are approximated according to the extrapolated values that we report and the order of convergence is affected for different polynomial approximations. More precisely,  all the computed orders of convergence are $2$ due to the variational crime that is committed when a curved domain is approximated by the straight lines of triangles.

\begin{figure}[H]
    \centering
    $\phi_1$\\
\includegraphics[width=0.22\textwidth]{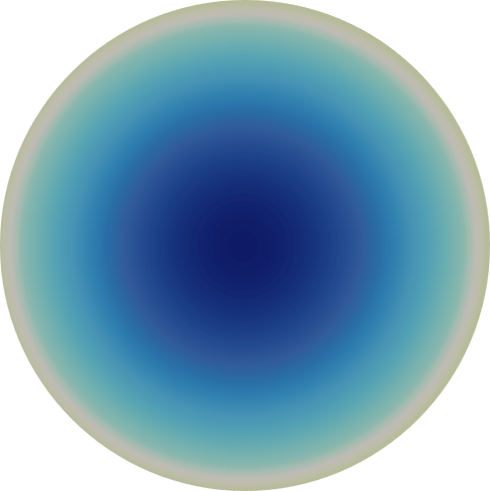}
    \includegraphics[width=0.22\textwidth]{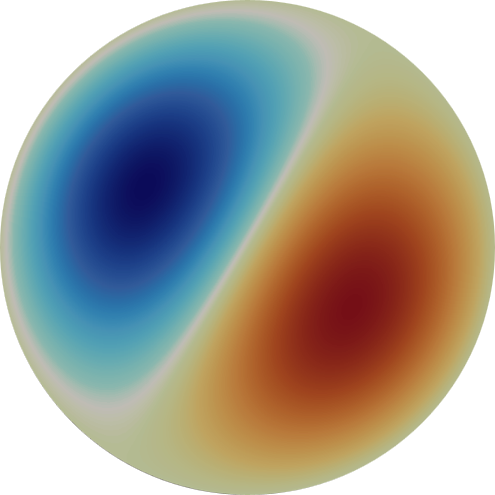}
    \includegraphics[width=0.22\textwidth]{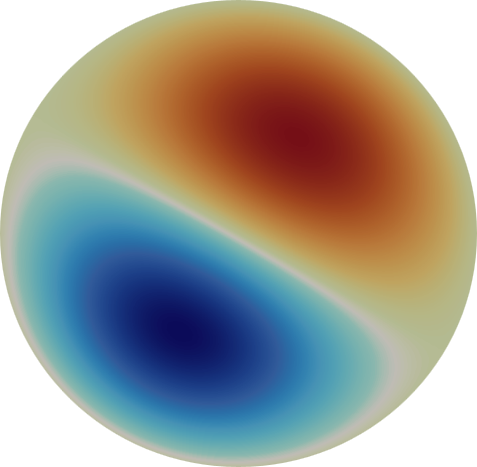}
    \includegraphics[width=0.22\textwidth]{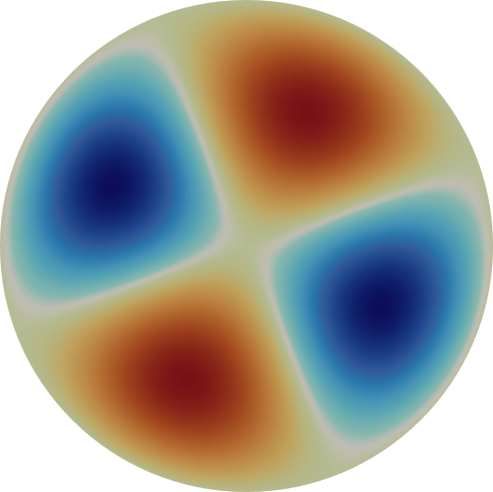}
    \includegraphics[width=0.07\textwidth]{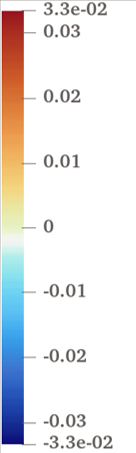}
    $\phi_2$\\
    \includegraphics[width=0.22\textwidth]{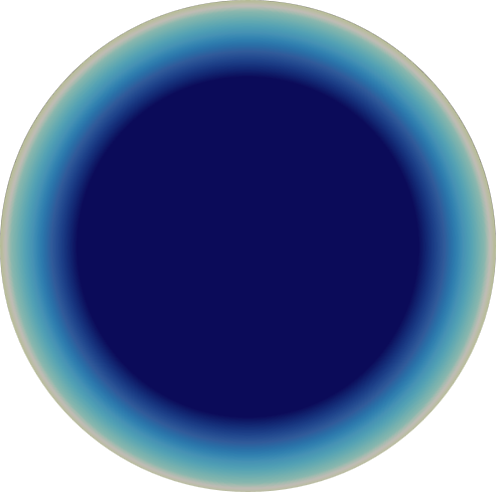}
    \includegraphics[width=0.22\textwidth]{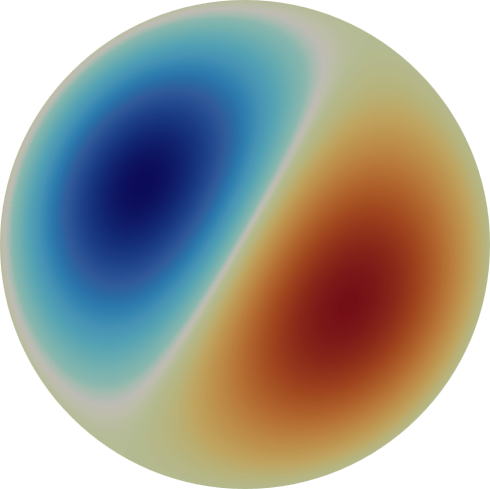}
    \includegraphics[width=0.22\textwidth]{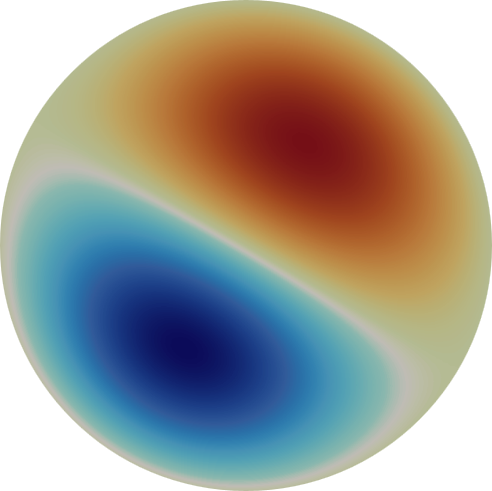}
    \includegraphics[width=0.22\textwidth]{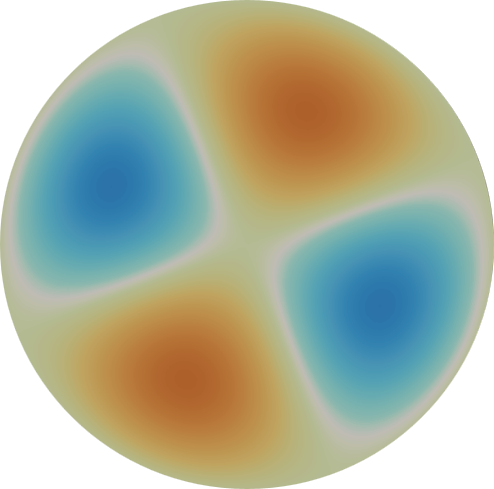}
    \includegraphics[width=0.07\textwidth]{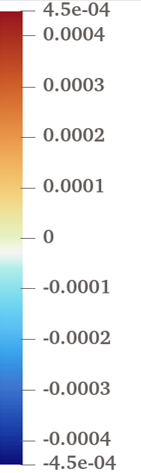}
  
    \caption{Eigenfunctions associated to the first  four eigenvalues on the unit circle d with $k=1$.}
    \label{fig:four_images_circle}
\end{figure}

\subsection{Unit cube with Dirichlet boundary conditions}
For this test we consider the unit cube $\O\coloneqq(0,1)^3$. Once again, the eigenfunctions result to be sufficiently smooth due to the convexity of $\O$ and the boundary conditions, implying an optimal order of convergence for the eigenvalues. For simplicity, we have considered only $k=1$ for the polynomial approximation due to the computational cost of solving an eigenvalue problem in 3D. In Table \ref{tab:eigen_convergence_cube} we present the results of our test where we report the optimal order of convergence for the eigenvalues.
\begin{table}[H]
    \centering
    \renewcommand{\arraystretch}{1.2}
    \setlength{\tabcolsep}{5pt}
    \begin{tabular}{c c c c c c}
        \toprule
        $N=8$ & $N=16$ & $N=32$ & $N=64$ & Order & $\lambda_{ext}$ \\
        \midrule
        125.7749 & 105.8681 & 101.0620 &  99.8701 & 2.04 & 99.4939 \\
        277.0978 & 217.4688 & 202.9682 & 199.3684 & 2.04 & 198.2791 \\
        277.0978 & 217.4688 & 202.9682 & 199.3684 & 2.04 & 198.2791 \\
        331.8180 & 229.4097 & 205.7724 & 200.0577 & 2.11 & 198.5133 \\
        487.7576 & 345.4144 & 309.0784 & 299.9281 & 1.97 & 296.6781 \\
        \bottomrule
    \end{tabular}
    \caption{Convergence behavior of the first five lowest computed eigenvalues on the unit cube domain, with homogeneous Dirichlet boundary conditions and polynomial degree $k=1$.}
    \label{tab:eigen_convergence_cube}
\end{table}

\subsection{Fuel Assembly simulation}
We conclude the numerical  section validating our method  with the IAEA 2D benchmark  given by \cite{iaea}, which consists on a quarter core model of a nuclear reactor. 

In Figure~\ref{fig:geo-iaea} we show the geometry used, which includes different material domains. In each domain we use different diffusion coefficients and cross sections, as shown in Table~\ref{tab:cross-ct}, together with the  required parameters.
\begin{figure}[H]
    \centering
    \begin{subfigure}[t]{0.4\textwidth}
        \centering
        \includegraphics[width=\textwidth]{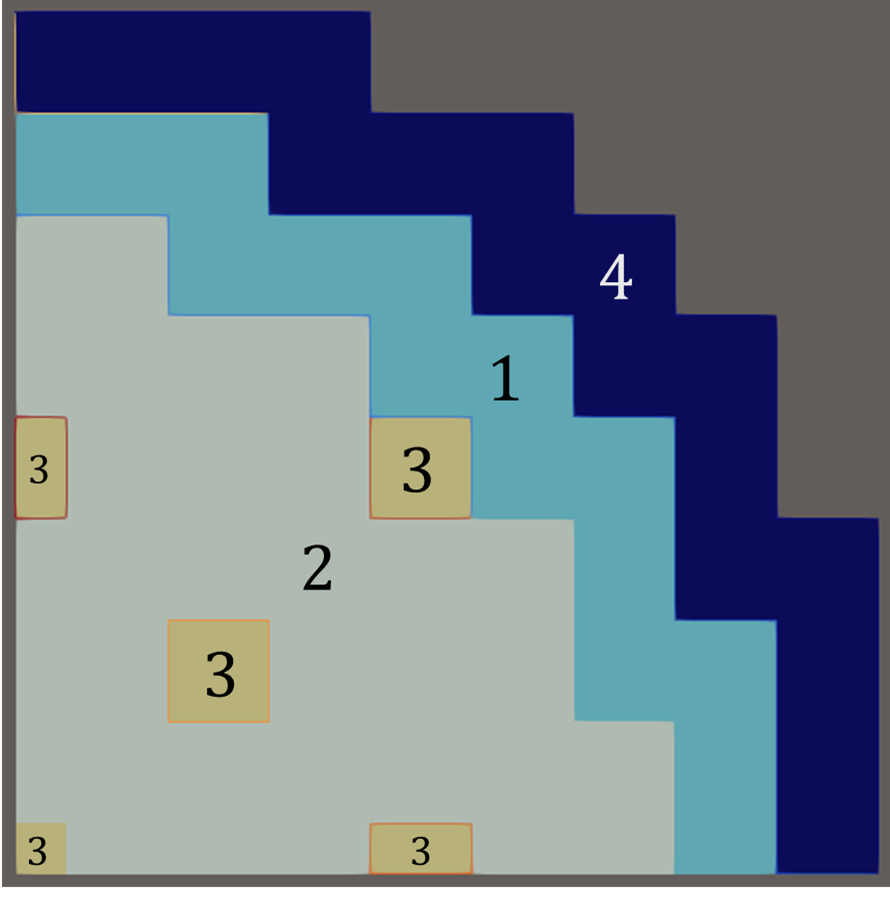}
        \caption{Horizontal Cross Section}
    \end{subfigure}
    \hspace{1.5cm} 
    \begin{subfigure}[t]{0.4\textwidth}
        \centering
        \includegraphics[width=0.5\textwidth]{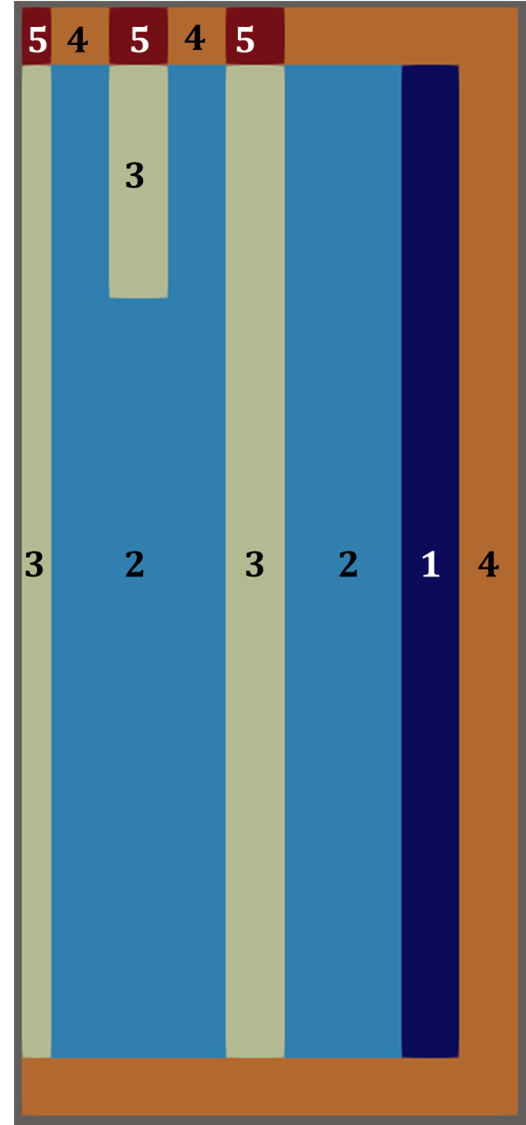}
        \caption{Vertical Cross Section}
    \end{subfigure}
    \caption{Results with IAEA 2D benchmark}
    \label{fig:geo-iaea}
\end{figure}
We impose  Robin boundary conditions for the external boundary, as follows,
\begin{equation*}
    \frac{\partial\phi_1}{\partial n} = \frac{-0.4692}{D_1}\phi_1\quad\text{and}\quad \frac{\partial\phi_2}{\partial n} = \frac{-0.4692}{D_2}\phi_2,\quad \text{on }\partial\Omega.
\end{equation*}

\begin{table}[H]
    \centering
    \begin{tabular}{c|c|c|c|c|c|c|l}\toprule
         Region&  $D_1$&  $D_2$&  $\Sigma_{1\to2}$&  $\Sigma_{a1}$&  $\Sigma_{a2}$& $\nu\Sigma_{f1}$ &$\nu\Sigma_{f2}$\\\midrule
         1& 1.5 & 0.4 & 0.02 & 0.01 & 0.080 & 0 & 0.135 \\
         2& 1.5 & 0.4 & 0.02 & 0.01 & 0.085 & 0 & 0.135 \\
         3& 1.5 & 0.4 & 0.02 & 0.01 & 0.130 & 0 & 0.135 \\
         4& 2.0 & 0.3 & 0.04 & 0    & 0.01  & 0 & 0 \\ 
 5& 2.0& 0.3& 0.04& 0& 0.055& 0&0\\\bottomrule
    \end{tabular}
    \caption{IAEA benchmark constants}
    \label{tab:cross-ct}
\end{table}
The aim of this Benchmark is to visualize the spacial distribution of the neutron fast and thermal flux. 

\begin{figure}[H]
    \centering
    \begin{subfigure}[t]{0.38\textwidth}
        \centering
        \includegraphics[width=\textwidth]{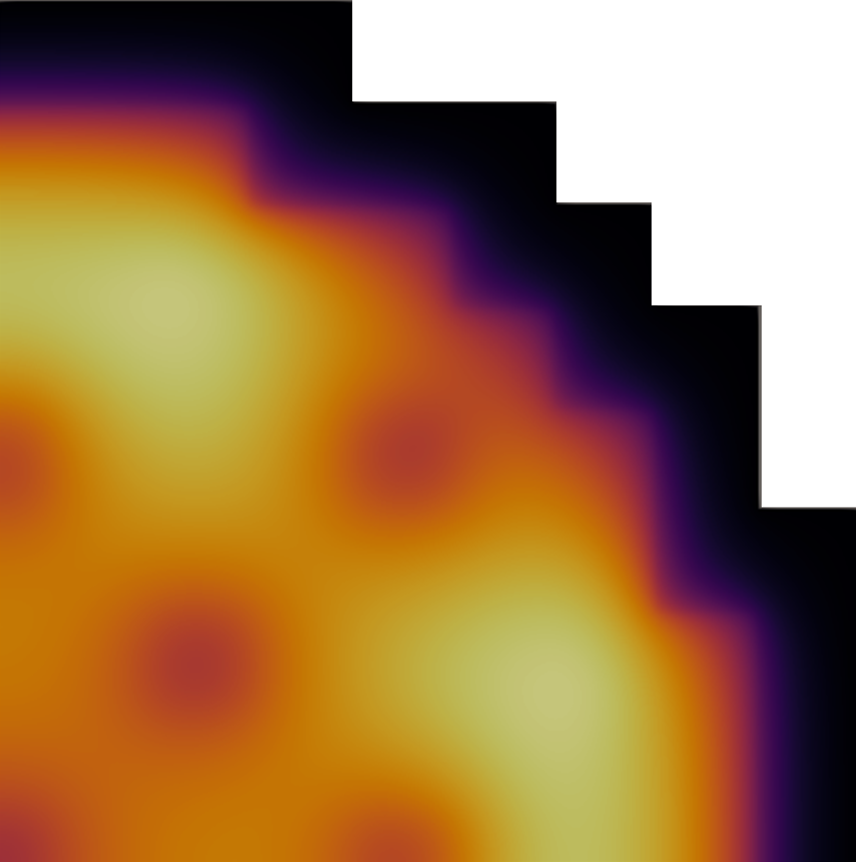}
        \caption{Fast group results}
    \end{subfigure}
    \hspace{1.5cm} 
    \begin{subfigure}[t]{0.38\textwidth}
        \centering
        \includegraphics[width=\textwidth]{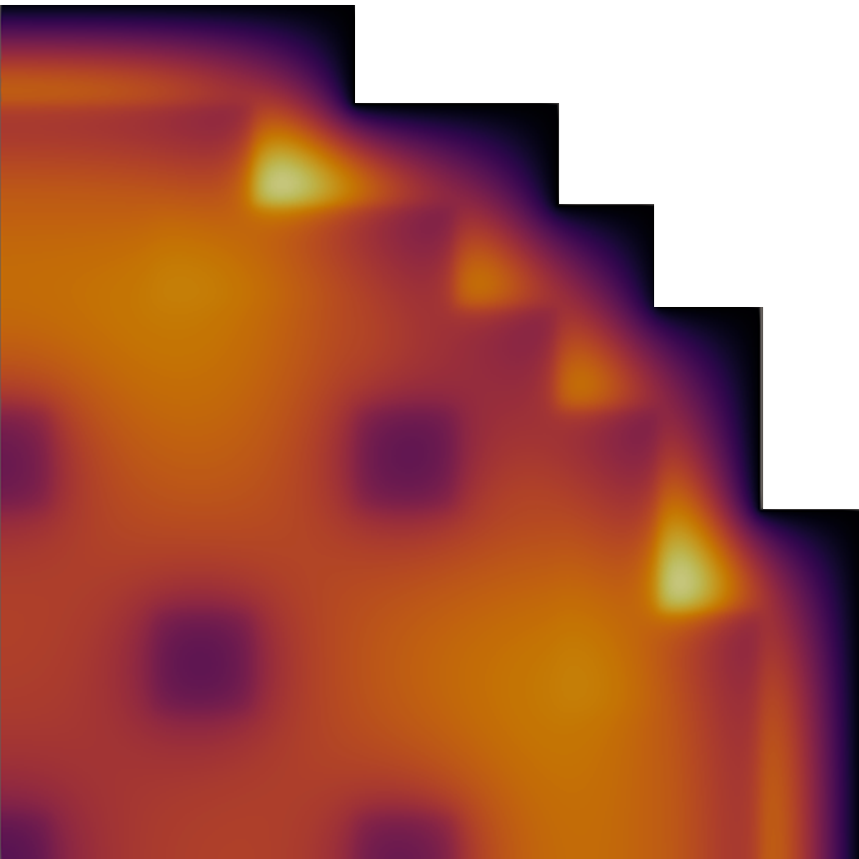}
        \caption{Thermal group results}
    \end{subfigure}
    \caption{Results with IAEA 2D benchmark for horizontal cross section}
    \label{fig:two_images_square_horizontal}
\end{figure}
\begin{figure}[H]
    \centering
    \begin{subfigure}[t]{0.4\textwidth}
        \centering
        \includegraphics[width=0.5\textwidth]{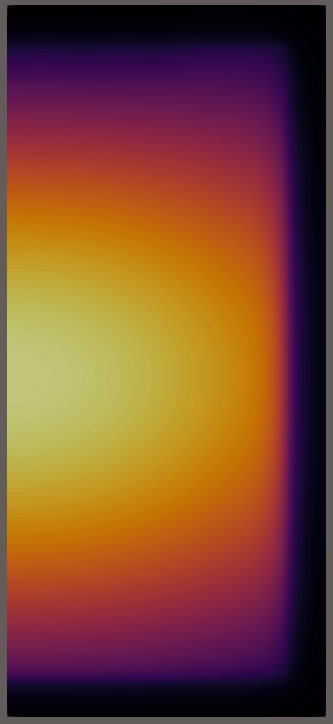}
        \caption{Fast group results}
    \end{subfigure}
    \hspace{1.5cm} 
    \begin{subfigure}[t]{0.4\textwidth}
        \centering
        \includegraphics[width=0.5\textwidth]{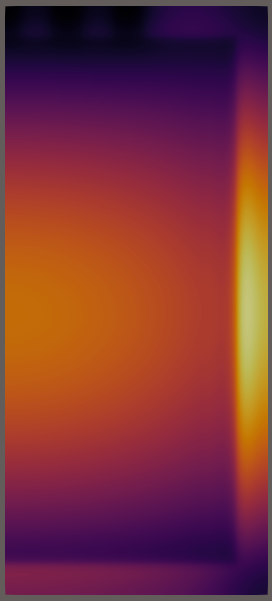}
        \caption{Thermal group results}
    \end{subfigure}
    \caption{Results with IAEA 2D benchmark for vertical cross section}
    \label{fig:two_images_square_vertical}
\end{figure}

With this simulation, we have obtained a multiplication factor $k = 0.9814$. This value is close to 1, indicating that the nuclear chain reaction is approximately stable (i.e., the system is near criticality). The results obtained are similar as the ones from literature, such as in \cite{iaea2}. These results indicate that both Fast and Thermal Neutron Fluxes reach higher concentration in the interior of the core that decreases towards the reactor boundary. The Fast Neutron shows a smooth and homogeneous  distribution, whereas the Thermal Neutron Flux is more heterogeneous and sows a higher sensitivity to the nuclear rods conditions.

\bibliographystyle{siamplain}
\bibliography{oseen-eigenvalue}

\appendix
\section{Fast Group derivation}
\label{appendix-a}
We show the derivation for the Fast Group Equation, starting from the general Neutron Transport Equation.

We will simplify each one of the  terms from the equation, 

\begin{equation*}
\begin{aligned}
&\underbrace{\int_{E_1}^{E_2}\Sigma_a(\epsilon)\,\phi(\vec{r},\epsilon)\,d\epsilon}_{\text{Absorption}}
+ \underbrace{\int_{E_1}^{E_2}\nabla_r \cdot D(\epsilon)\nabla_r\phi(\vec{r},\epsilon)\,d\epsilon}_{\text{Leakage}} 
+ \underbrace{\int_{E_1}^{E_2}\int_{E}\Sigma_s(\epsilon\to\epsilon')\phi(\vec{r},\epsilon)\,d\epsilon'd\epsilon}_{\text{Out-scattering}} \\
&= \underbrace{\int_{E_1}^{E_2}\int_{E}\Sigma_s(\epsilon'\to\epsilon)\phi(\vec{r},\epsilon')\,d\epsilon'}_{\text{In-scattering}} + \underbrace{\int_{E_1}^{E_2}\chi(\epsilon)\,d\epsilon\,\frac{1}{k}\int_{E}\nu(\epsilon')\Sigma_f(\epsilon') \phi(\vec{r},\epsilon')\,d\epsilon'}_{\text{Fission}} .
\end{aligned}
\end{equation*}

 by using the constant cross sections defined in \ref{constant-cross}.

\paragraph{Absorption Term}
\begin{equation}
\int_{E_1}^{E_2}\Sigma_a(\epsilon)\phi(\vec{r},\epsilon)\,d\epsilon=\frac{\displaystyle\int_{E_1}^{E_2}\Sigma_a(\epsilon) \phi(\vec{r},\epsilon)\,d\epsilon}{\displaystyle\int_{E_1}^{E_2}\phi(\vec{r},\epsilon)\,d\epsilon}\int_{E_1}^{E_2}\phi(\vec{r},\epsilon)\,d\epsilon = \Sigma_{f1}\phi_1.
\end{equation}
\paragraph{Leakage Term}

\begin{equation}
\begin{aligned}
    \int_{E_1}^{E_2}\nabla_r \cdot (D(\epsilon)\nabla_r\phi(\vec{r},\epsilon))\,d{\epsilon} &= \nabla_r \cdot  \int_{E_1}^{E_2}D(\epsilon)\nabla_r\phi(\vec{r},\epsilon)\,d{\epsilon}\\
    &=\nabla_r \cdot  \frac{\displaystyle\int_{E_1}^{E_2}D(\epsilon)\nabla_r\phi(\vec{r},\epsilon)\,d{\epsilon}}{\displaystyle\int_{E_1}^{E_2}\nabla_r\phi(\vec{r},\epsilon)\,d{\epsilon}}\int_{E_1}^{E_2}\nabla_r\phi(\vec{r},\epsilon)\,d{\epsilon}\\
    &=\nabla_r \cdot  D_1\nabla_r\phi_1.
\end{aligned}
\end{equation}
\paragraph{Out-Scattering Term}
\begin{equation}
\begin{aligned}
   & \displaystyle\int_{E_1}^{E_2}\int_{E }\Sigma_s(\epsilon\to\epsilon')\phi(\vec{r},\epsilon)\,d\epsilon'd\epsilon \\&= \displaystyle\int_{E_1}^{E_2}\int_{E_0 }^{E_1}\Sigma_s(\epsilon\to\epsilon')\phi(\vec{r},\epsilon)\,d\epsilon'd\epsilon+\int_{E_1}^{E_2}\displaystyle\int_{E_1 }^{E_2}\Sigma_s(\epsilon\to\epsilon')\phi(\vec{r},\epsilon)\,d\epsilon'd\epsilon \\
    &= \frac{\displaystyle\int_{E_1}^{E_2}\int_{E_0 }^{E_1}\Sigma_s(\epsilon\to\epsilon')\phi(\vec{r},\epsilon)\,d\epsilon'd\epsilon}{\displaystyle\int_{E_1}^{E_2}\phi(\vec{r},\epsilon)\,d\epsilon}\displaystyle\int_{E_1}^{E_2}\phi(\vec{r},\epsilon)\,d\epsilon'd\epsilon\\&\quad+\frac{\displaystyle\int_{E_1}^{E_2}\int_{E_1 }^{E_2}\Sigma_s(\epsilon\to\epsilon')\phi(\vec{r},\epsilon)\,d\epsilon'd\epsilon}{\displaystyle\int_{E_1}^{E_2}\phi(\vec{r},\epsilon)\,d\epsilon}\int_{E_1}^{E_2}\phi(\vec{r},\epsilon)\,d\epsilon'd\epsilon\\
    &= \Sigma_{1\to2}\phi_1 + \Sigma_{1\to1}\phi_1.
\end{aligned}
\end{equation}
\paragraph{Fission Term}
\begin{footnotesize}
\begin{align*}
    &\int_{E_1}^{E_2}\chi(\epsilon)\,d\epsilon\,\frac{1}{k}\int_{E}\nu(\epsilon')\Sigma_f(\epsilon') \phi(\vec{r},\epsilon')\,d\epsilon' \\&= \frac{1}{k}\int_{E}\nu(\epsilon')\Sigma_f(\epsilon') \phi(\vec{r},\epsilon')\,d\epsilon'\\
    &=\frac{1}{k}\left(\int_{E_0}^{E_1}\nu(\epsilon')\Sigma_f(\epsilon') \phi(\vec{r},\epsilon')\,d\epsilon'+\int_{E_1}^{E_2}\nu(\epsilon')\Sigma_f(\epsilon') \phi(\vec{r},\epsilon')\,d\epsilon'\right)\\
    &=\frac{1}{k}\left(\dfrac{\displaystyle\int_{E_0}^{E_1}\nu(\epsilon')\Sigma_f(\epsilon') \phi(\vec{r},\epsilon')\,d\epsilon'}{\int_{E_0}^{E_1}\phi(\vec{r},\epsilon)d\epsilon}\displaystyle\int_{E_0}^{E_1}\phi(\vec{r},\epsilon)d\epsilon+\frac{\displaystyle\int_{E_1}^{E_2}\nu(\epsilon')\Sigma_f(\epsilon') \phi(\vec{r},\epsilon')\,d\epsilon'}{\displaystyle\int_{E_1}^{E_2}\phi(\vec{r},\epsilon)d\epsilon}\displaystyle\int_{E_1}^{E_2}\phi(\vec{r},\epsilon)d\epsilon\right)\\
    &=\frac{1}{k}\left( \nu_1\Sigma_{f1}\phi_1+\nu_2\Sigma_{f2}\phi_2\right).
    \end{align*}
\end{footnotesize}
    
    \paragraph{In-Scattering Term}
\begin{equation}
\begin{aligned}
   & \displaystyle\int_{E_1}^{E_2}\int_{E }\Sigma_s(\epsilon'\to\epsilon)\phi(\vec{r},\epsilon')\,d\epsilon'd\epsilon \\&= \displaystyle\int_{E_1}^{E_2}\int_{E_0 }^{E_1}\Sigma_s(\epsilon'\to\epsilon)\phi(\vec{r},\epsilon')\,d\epsilon'd\epsilon+\displaystyle\int_{E_1}^{E_2}\displaystyle\int_{E_1 }^{E_2}\Sigma_s(\epsilon'\to\epsilon)\phi(\vec{r},\epsilon')\,d\epsilon'd\epsilon \\
    &= \frac{\displaystyle\int_{E_1}^{E_2}\int_{E_0 }^{E_1}\Sigma_s(\epsilon'\to\epsilon)\phi(\vec{r},\epsilon')\,d\epsilon'd\epsilon}{\int_{E_1}^{E_2}\phi(\vec{r},\epsilon)\,d\epsilon}\int_{E_1}^{E_2}\phi(\vec{r},\epsilon)\,d\epsilon'd\epsilon\\&\quad+\frac{\displaystyle\int_{E_1}^{E_2}\displaystyle\int_{E_1 }^{E_2}\Sigma_s(\epsilon'\to\epsilon)\phi(\vec{r},\epsilon')\,d\epsilon'd\epsilon}{\displaystyle\int_{E_1}^{E_2}\phi(\vec{r},\epsilon)\,d\epsilon}\displaystyle\int_{E_1}^{E_2}\phi(\vec{r},\epsilon)\,d\epsilon'd\epsilon\\
    &= \Sigma_{2\to1}\phi_1 + \Sigma_{1\to1}\phi_1,
\end{aligned}
\end{equation}
This will give us the following equivalence,
    \begin{multline*}
    -\nabla\cdot(D_1\nabla \phi_1) + \Sigma_{a1}\phi_1+\Sigma_{1\to2}\phi_1 +\Sigma_{1\to1}\phi_1\\
    = \Sigma_{2\to1}\phi_1+\Sigma_{1\to1}\phi_1+\frac{1}{k}(\nu_1\Sigma_{f1}\phi_1+\nu_2\Sigma_{f2}\phi_2).
\end{multline*}
We cancel at each side the term $\Sigma_{1\to1}\phi_1$ and as we said before, we neglect the term $\Sigma_{2\to1}$, this give us the following equation for the fast neutron group.
\begin{equation}
    -\nabla\cdot(D_1\nabla \phi_1) + (\Sigma_{a1}+\Sigma_{1\to2})\phi_1 = \frac{1}{k}(\nu_1\Sigma_{f1}\phi_1+\nu_2\Sigma_{f2}\phi_2).
\end{equation}

\newpage
\section{Thermal Group derivation}
\label{appendix-b}
We show the derivation for the Fast Group Equation, starting from the general Neutron Transport Equation.

We will simplify each one of the  terms from the equation, 
\begin{equation*}
\begin{aligned}
&\underbrace{\int_{E_0}^{E_1}\Sigma_a(\epsilon)\,\phi(\vec{r},\epsilon)\,d\epsilon}_{\text{Absorption}}
+ \underbrace{\int_{E_0}^{E_1}\nabla \cdot D(\epsilon)\nabla\phi(\vec{r},\epsilon)\,d\epsilon}_{\text{Leakage}}+ \underbrace{\int_{E_0}^{E_1}\int_{E}\Sigma_s(\epsilon\to\epsilon')\phi(\vec{r},\epsilon)\,d\epsilon'd\epsilon}_{\text{Out-Scattering}} \\
&= \underbrace{\int_{E_0}^{E_1}\int_{E}\Sigma_s(\epsilon'\to\epsilon)\phi(\vec{r},\epsilon')\,d\epsilon'}_{\text{In-Scattering}}+ \underbrace{\int_{E_0}^{E_1}\chi(\epsilon)\,d\epsilon \,\frac{1}{k}\int_{E}\nu(\epsilon')\Sigma_f(\epsilon')\,\phi(\vec{r},\epsilon')\,d\epsilon'}_{\text{Fission}} .
\end{aligned}
\end{equation*}
by using the constant cross sections defined in \ref{constant-cross}.
\paragraph{Absorption Term}
\begin{equation}
\int_{E_0}^{E_1}\Sigma_a(\epsilon)\phi(\vec{r},\epsilon)\,d\epsilon=\frac{\displaystyle\int_{E_0}^{E_1}\Sigma_a(\epsilon) \phi(\vec{r},\epsilon)\,d\epsilon}{\displaystyle\int_{E_0}^{E_1}\phi(\vec{r},\epsilon)\,d\epsilon}\displaystyle\int_{E_0}^{E_1}\phi(\vec{r},\epsilon)\,d\epsilon = \Sigma_{f2}\phi_2.
\end{equation}
\paragraph{Leakage Term}
\begin{equation}
\begin{aligned}
    \int_{E_0}^{E_1}\nabla \cdot D(\epsilon)\nabla\phi(\vec{r},\epsilon)\,d{\epsilon} &= \nabla \cdot  \int_{E_0}^{E_1}D(\epsilon)\nabla\phi(\vec{r},\epsilon)\,d{\epsilon}\\
    &=\nabla \cdot  \frac{\displaystyle\int_{E_0}^{E_1}D(\epsilon)\nabla\phi(\vec{r},\epsilon)\,d{\epsilon}}{\displaystyle\int_{E_0}^{E_1}\nabla\phi(\vec{r},\epsilon)\,d{\epsilon}}\int_{E_0}^{E_1}\nabla\phi(\vec{r},\epsilon)\,d{\epsilon}\\
    &=\nabla \cdot  D_2\nabla\phi_2.
\end{aligned}
\end{equation}
\paragraph{Out-Scattering Term}
\begin{equation}
\begin{aligned}
   & \displaystyle\int_{E_0}^{E_1}\int_{E }\Sigma_s(\epsilon\to\epsilon')\phi(\vec{r},\epsilon)\,d\epsilon'd\epsilon \\&= \displaystyle\int_{E_0}^{E_1}\int_{E_0 }^{E_1}\Sigma_s(\epsilon\to\epsilon')\phi(\vec{r},\epsilon)\,d\epsilon'd\epsilon+\displaystyle\int_{E_0}^{E_1}\int_{E_1 }^{E_2}\Sigma_s(\epsilon\to\epsilon')\phi(\vec{r},\epsilon)\,d\epsilon'd\epsilon \\
    &= \frac{\displaystyle\int_{E_0}^{E_1}\int_{E_0 }^{E_1}\Sigma_s(\epsilon\to\epsilon')\phi(\vec{r},\epsilon)\,d\epsilon'd\epsilon}{\displaystyle\int_{E_0}^{E_1}\phi(\vec{r},\epsilon)\,d\epsilon}\displaystyle\int_{E_0}^{E_1}\phi(\vec{r},\epsilon)\,d\epsilon'd\epsilon\\&\quad+\frac{\displaystyle\int_{E_0}^{E_1}\displaystyle\int_{E_1 }^{E_2}\Sigma_s(\epsilon\to\epsilon')\phi(\vec{r},\epsilon)\,d\epsilon'd\epsilon}{\displaystyle\int_{E_0}^{E_1}\phi(\vec{r},\epsilon)\,d\epsilon}\int_{E_0}^{E_1}\phi(\vec{r},\epsilon)\,d\epsilon'd\epsilon\\
    &= \Sigma_{2\to2}\phi_2 + \Sigma_{2\to1}\phi_2.
\end{aligned}
\end{equation}
\paragraph{Fission Term}

\begin{align}
    \int_{E_0}^{E_1}\chi(\epsilon)\,d\epsilon\,\frac{1}{k}\int_{E}\nu(\epsilon')\Sigma_f(\epsilon') \phi(\vec{r},\epsilon')\,d\epsilon' = 0.
\end{align}

    \paragraph{In-Scattering Term}
\begin{equation}
\begin{aligned}
   & \int_{E_0}^{E_1}\int_{E }\Sigma_s(\epsilon'\to\epsilon)\phi(\vec{r},\epsilon')\,d\epsilon'd\epsilon \\&= \int_{E_0}^{E_1}\int_{E_0 }^{E_1}\Sigma_s(\epsilon'\to\epsilon)\phi(\vec{r},\epsilon')\,d\epsilon'd\epsilon+\int_{E_0}^{E_1}\int_{E_1 }^{E_2}\Sigma_s(\epsilon'\to\epsilon)\phi(\vec{r},\epsilon')\,d\epsilon'd\epsilon \\
    &= \frac{\displaystyle\int_{E_0}^{E_1}\int_{E_0 }^{E_1}\Sigma_s(\epsilon'\to\epsilon)\phi(\vec{r},\epsilon')\,d\epsilon'd\epsilon}{\displaystyle\int_{E_0}^{E_1}\phi(\vec{r},\epsilon)\,d\epsilon}\displaystyle\int_{E_0}^{E_1}\phi(\vec{r},\epsilon)\,d\epsilon'd\epsilon\\&\quad+\frac{\displaystyle\int_{E_0}^{E_1}\int_{E_1 }^{E_2}\Sigma_s(\epsilon'\to\epsilon)\phi(\vec{r},\epsilon')\,d\epsilon'd\epsilon}{\displaystyle\int_{E_0}^{E_1}\phi(\vec{r},\epsilon)\,d\epsilon}\int_{E_0}^{E_1}\phi(\vec{r},\epsilon)\,d\epsilon'd\epsilon\\
    &= \Sigma_{2\to2}\phi_2 + \Sigma_{1\to2}\phi_1.
\end{aligned}
\end{equation}
This will give us the following equivalency,
\begin{equation*}
    -\nabla\cdot(D_2\nabla \phi_2) + \Sigma_{a2}\phi_2-\Sigma_{1\to2}\phi_1 = 0.
\end{equation*}
\section{Real Eigenvalues Lemma}
\label{apendix-c}
\begin{lemma}
\label{lemma-apc}
    Let $X$ be a Hilbert space and let  $\mathcal{L}$ and $\mathcal{R}$ be self-adjoint operators defined from $X$ to $X'$. Consider the generalized eigenvalue problem: find $(\lambda, u)\in\mathbb{C}\times X$, with $u\neq 0$,  such that, 
    \begin{equation}
    \label{eq:LR_eigen}
        \mathcal{L}x = \lambda\mathcal{R}x.
    \end{equation} If one of them, $\mathcal{L}$ or $\mathcal{R}$ is positive definite, then  $\lambda\in\R$. Moreover, if both of them are positive definite, then $\lambda>0$.

    \begin{proof}
    Let us consider the   sesquilinear forms induced by  operators $\mathcal{L}$ and $\mathcal{R}$, defined as follows
    \begin{equation*}
        \ell(u,v) \coloneqq\langle \mathcal{L}u,v\rangle\quad\text{and}\quad
        r(u,v)  \coloneqq\langle \mathcal{R}u,v\rangle.
    \end{equation*}


    With these sesquilinear forms at hand, we rewrite problem \eqref{eq:LR_eigen} as follows: find  $(\lambda,u)\in\mathbb{C}\times X$, with $u\neq 0$,  such that 
    \begin{equation*}
        \ell(u,v) = \lambda \,r(u,v),\quad \forall v\in X.
    \end{equation*}
    Choosing $v = u$ and using the hermitian properties of both sesquilinear forms, it holds that $\ell(u,u) = \overline{\ell (u,u)}$ and $r(u,u) = \overline{r(u,u)}$, so that $\ell(u,u), r(u,u)  \in \mathbb{R}$. If we consider $r(u,v)$ positive definite, then it is possible to write $ \lambda =\ell(u,u)/r(u,u)$,
    so that $\lambda \in \R$. If both $\ell(u,v)$ and $r(u,v)$ are positive definite, it holds that $\lambda > 0$.
\end{proof}
\end{lemma}
\end{document}